\documentclass[12pt]{amsart}

\usepackage[margin=1in]{geometry}

\usepackage{amsthm,amsmath,amsfonts,amssymb}
\usepackage{ytableau}
\ytableausetup{smalltableaux}
\usepackage{enumerate}

\usepackage{graphicx}

\usepackage{tikz, tikz-3dplot, pgfplots}
\usepackage{tkz-graph}
\usetikzlibrary{matrix}
\usetikzlibrary[positioning,patterns] % tikz libraries for relative positioning and silly patterns
\usepackage{color}

\newtheorem{thm}{Theorem}[section]

\newtheorem{cor}[thm]{Corollary}
\newtheorem{lemma}[thm]{Lemma}
\theoremstyle{definition}
\newtheorem{example}[thm]{Example}
\newtheorem{defn}[thm]{Definition}

\newcommand{\N}{\mathbb{N}}
\newcommand{\Z}{\mathbb{Z}}
\newcommand{\x}{\mathbf{x}}
\newcommand{\C}{\mathbb{C}}
\newcommand{\QSym}{\mathrm{QSym}}
\newcommand{\KP}{K_P(\x)}
\newcommand{\KQ}{K_Q(\x)}
\newcommand{\Comp}{\mathrm{Comp}}
\newcommand{\des}{\mathrm{des}}

\DeclareMathOperator{\anti}{anti}
\DeclareMathOperator{\rank}{rank}
\DeclareMathOperator{\co}{co}
\DeclareMathOperator{\sh}{sh}
\DeclareMathOperator{\supp}{supp}
\newcommand{\antikij}{\anti_{k,i,j}}

\DeclareMathOperator{\jumppair}{jumppair}
\DeclareMathOperator{\jump}{jump}
\DeclareMathOperator{\upjump}{up-jump}
\newcommand{\V}{V}

\usepackage[backend=bibtex]{biblatex}
\title{$P$-Partition Generating Function Equivalence of Naturally Labeled Posets}

\author{Ricky Ini Liu}
\address{Department of Mathematics, North Carolina State University, Raleigh, NC}
\email{riliu@ncsu.edu}
\author{Michael Weselcouch}
\address{Department of Mathematics, North Carolina State University, Raleigh, NC}
\email{mweselc@ncsu.edu}

\thanks{R. I. Liu and M. Weselcouch were partially supported by National Science Foundation grant DMS-1700302.}

\date{\today}

\begin{document}
\maketitle

\begin{abstract}
The \emph{$P$-partition generating function} of a (naturally labeled) poset $P$ is a quasisymmetric function enumerating order-preserving maps from $P$ to $\Z^+$. Using the Hopf algebra of posets, we give necessary conditions for two posets to have the same generating function. In particular, we show that they must have the same number of antichains of each size, as well as the same shape (as defined by Greene). We also discuss which shapes guarantee uniqueness of the $P$-partition generating function and give a method of constructing pairs of non-isomorphic posets with the same generating function.
\end{abstract}

\section{Introduction}
For a finite poset $(P, \prec)$ (labeled with the ground set $[n] = \{1, 2, \dots n \}$), the \emph{$P$-partition generating function} $\KP$ is a quasisymmetric function enumerating certain order-preserving maps from $P$ to $\Z^+$. The question of when two distinct posets can have the same $P$-partition generating function has been studied extensively in the case of skew Schur functions \cite{BTvW, MvW, RSvW}, by McNamara and Ward \cite{McNamaraWard} for general labeled posets, and by Hasebe and Tsujie \cite{HasebeTsujie} for rooted trees. The goal of this paper is to consider the naturally labeled case, that is, to give necessary and sufficient conditions for when two naturally labeled posets have the same $P$-partition generating function. (We say that $P$ is \emph{naturally labeled} if $x \preceq y$ implies $x \leq y$ as integers.)

In general, it is not true that a poset can be distinguished by its $P$-partition generating function. The smallest case in which two distinct naturally labeled posets have the same partition generating function is the two $7$-element posets shown below. We will explore this example further in Section 5, where we give a general construction for non-isomorphic posets with the same generating function.

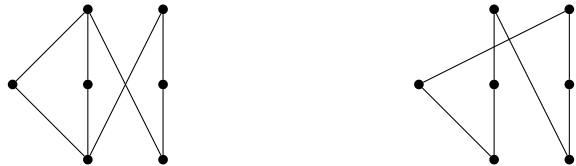
\begin{center}
\begin{tikzpicture}
 [auto,
 vertex/.style={circle,draw=black!100,fill=black!100,thick,inner sep=0pt,minimum size=1mm}]
\node (v1) at ( 0,0) [vertex] {};
\node (v2) at ( 1,0) [vertex] {};
\node (v3) at ( -1,1) [vertex] {};
\node (v4) at ( 0,1) [vertex] {};
\node (v5) at ( 1,1) [vertex] {};
\node (v6) at ( 0,2) [vertex] {};
\node (v7) at ( 1,2) [vertex] {};
\draw [-] (v1) to (v4);
\draw [-] (v4) to (v6);
\draw [-] (v2) to (v5);
\draw [-] (v5) to (v7);
\draw [-] (v1) to (v3);
\draw [-] (v3) to (v6);
\draw [-] (v1) to (v7);
\draw [-] (v2) to (v6);
\end{tikzpicture}
\hspace{3 cm}
\begin{tikzpicture}
 [auto,
 vertex/.style={circle,draw=black!100,fill=black!100,thick,inner sep=0pt,minimum size=1mm}]
\node (v1) at ( 4,0) [vertex] {};
\node (v2) at ( 5,0) [vertex] {};
\node (v3) at ( 3,1) [vertex] {};
\node (v4) at ( 4,1) [vertex] {};
\node (v5) at ( 5,1) [vertex] {};
\node (v6) at ( 4,2) [vertex] {};
\node (v7) at ( 5,2) [vertex] {};
\draw [-] (v1) to (v4);
\draw [-] (v4) to (v6);
\draw [-] (v2) to (v5);
\draw [-] (v5) to (v7);
\draw [-] (v1) to (v3);
\draw [-] (v3) to (v7);
\draw [-] (v2) to (v6);
\end{tikzpicture}
\end{center}  

We will use tools from the combinatorial Hopf algebra structure on posets due to Schmitt \cite{Schmitt} (see also \cite{ABS}) to prove that if $\KP = \KQ$, then for all triples $(k, i, j)$, $P$ and $Q$ must have the same number of $k$-element order ideals that have $i$ maximal elements and whose complement has $j$ minimal elements. In particular, they must have the same number of antichains of each size, proving a conjecture of McNamara and Ward \cite{McNamaraWard}. As a result of our proof, one can compute certain coefficients in the fundamental quasisymmetric function expansion of $\KP$ explicitly in terms of the number of such ideals.

We will also show that if $\KP = \KQ$, then $P$ and $Q$ must have the same shape. Here, the \emph{shape} of a finite poset, denoted $\sh(P)$, is the partition $\lambda$ whose conjugate partition $\lambda'$ satisfies
\[\lambda'_1 + \lambda'_2+\cdots+\lambda'_i = a_i,\]
where $a_i$ is the largest number of elements in a union of $i$ antichains of $P$. In fact, we will prove a stronger statement, namely that if the support of $\KP$ and $\KQ$ in the fundamental quasisymmetric function basis is the same, then $P$ and $Q$ must have the same shape. This suggests the following question: for which partitions $\lambda$ does $\sh(P) = \lambda$ guarantee that $P$ is uniquely determined by $\KP$?

We show that if $\sh(P)$ has at most two parts, is a hook shape, or has the form $\sh(P) = (\lambda_1, 2, 1, \dots, 1)$, then $\KP = \KQ$ implies $P \cong Q$. Conversely, we show that if $\sh(P)$ contains $(3, 3, 1)$ or $(2, 2, 2, 2)$, then $\KP = \KQ$ does not necessarily imply $P \cong Q$ by constructing two distinct posets of this shape with the same generating function.  It remains to be answered what happens when $\sh(P) = (\lambda_1, 2, 2, 1, \dots, 1)$.

In Section 2 we will give some preliminary information; in Section 3 we state some necessary conditions for two posets to have the same generating function; in Section 4 we discuss when the shape of a poset ensures that its generating function is unique; and in Section 5 we give a general construction for pairs of posets with the same generating function.

\section{Preliminaries}

We begin with some preliminaries about posets, quasisymmetric functions, and Hopf algebras. For more information, see \cite{GrinbergReiner, McNamaraWard, Stanley}.

\subsection{Posets and $P$-partitions}

Let $P = (P, \prec)$ be a finite poset. A \emph{labeling} of $P$ is a bijection $\omega \colon P \to \{1, 2, \dots, n\}$.

\begin{defn} For a labeled poset $(P, \omega)$, a \emph{$(P, \omega)$-partition} is a map $\sigma \colon P \to \Z^+$ that satisfies the following:
\begin{enumerate}[(a)]
\item If $x \preceq y$, then $\sigma(x) \leq \sigma(y)$.
\item If $x \preceq y$ and $\omega(x) > \omega(y)$, then $\sigma(x) < \sigma(y)$.
\end{enumerate}
\end{defn}

\begin{defn} The \emph{$(P, \omega)$-partition generating function} $K_{(P, \omega)}(x_1, x_2, \dots)$ for a labeled poset $(P, \omega)$ is given by 
\[ K_{(P, \omega)}(x_1, x_2, \dots ) = \sum_{(P, \omega)\text{-partition } \sigma}x_1^{|\sigma^{-1}(1)|}x_2^{|\sigma^{-1}(2)|}\dots,\]
where the sum ranges over all $(P, \omega)$-partitions $\sigma$.
\end{defn}
A labeled poset $(P, \omega)$ is equivalent to a poset $P$ with ground set $[n]$. Hence we may refer to the generating function $K_{(P, \omega)}(x_1, x_2, \dots )$ as $K_{P}(x_1, x_2, \dots )$ or $K_{P}(\mathbf{x})$ if the choice of $\omega$ is implicit.

In this paper, we will usually restrict our attention to the case when $P$ is $\emph{naturally labeled}$, that is, when $\omega$ is an order-preserving map. In this case, $\KP$ does not depend on our choice of natural labeling but only on the underlying structure of $P$.

A \emph{linear extension} of a poset $P$ with ground set $[n]$ is a permutation $\pi$ of $[n]$ that respects the relations in $P$, that is, if $x \preceq y$, then $\pi^{-1}(x) \leq \pi^{-1}(y)$.  The set of all linear extensions of $P$ is denoted $\mathcal{L}(P)$. Note that $|\mathcal L(P)|$ is the coefficient of $x_1x_2\cdots x_n$ in $\KP$.

\subsection{Compositions}

A \textit{composition} $\alpha= (\alpha_1, \alpha_2, \dots, \alpha_k)$ of $n$ is a finite sequence of positive integers summing to $n$. We denote the set of all compositions of $n$ by $\Comp_n$.  If $\alpha$ is a composition of $n$, then we write $|\alpha| = n$. The compositions of $n$ are in bijection with the subsets of $[n-1]$ in the following way: for any composition $\alpha$, define \[D(\alpha) = \{\alpha_1, \quad \alpha_1 + \alpha_2, \quad \dots, \quad \alpha_1 + \alpha_2 +\dots +\alpha_{k-1}\} \subseteq [n-1].\]

Likewise, for any subset $S = \{s_1, s_2, \dots, s_{k-1}\}\subseteq [n-1]$ with $s_1<s_2<\dots < s_{k-1}$, we can define the composition
\[\co(S) = (s_1, \quad s_2-s_1, \quad s_3-s_2, \quad \dots, \quad s_{k-1}-s_{k-2}, \quad n-s_{k-1}).\]

Given two nonempty compositions $\alpha = (\alpha_1, \alpha_2, \dots, \alpha_k)$ and $\beta = (\beta_1, \beta_2, \dots, \beta_m)$, their \emph{concatenation} is
\[\alpha \cdot \beta = (\alpha_1, \alpha_2, \dots, \alpha_k, \beta_1, \beta_2, \dots, \beta_m),\]
and their \emph{near-concatenation} is
\[ \alpha \odot \beta = (\alpha_1, \alpha_2, \dots, \alpha_k + \beta_1, \beta_2, \dots, \beta_m). \]
We will use the shorthand $1^k$ to denote the composition $(\underbrace{1, 1, \dots, 1}_k)$.

The \emph{ribbon representation} for $\alpha$ is the diagram having rows of sizes $(\alpha_1, \dots, \alpha_k)$ read from bottom to top with exactly one column of overlap between adjacent rows. For example, the figure below depicts the ribbon representation of $\alpha = (3, 1, 2, 4)$.
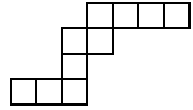
\begin{center}
\ydiagram{3+4,2+2,2+1,3}
\end{center}
Each composition can be written as the near-concatenation of compositions of all $1$s.  The composition $\alpha$ can be expressed as $\alpha = 1^{a_1} \odot 1^{a_2} \odot \dots \odot 1^{a_l}$, where $a_i$ is the number of boxes in the $i$th column of $\alpha$'s ribbon representation.  We will refer to the expansion $\alpha = 1^{a_1} \odot 1^{a_2} \odot \dots \odot 1^{a_l}$ as the \emph{near-concatenation decomposition} of $\alpha$.

A \emph{partition} $\lambda = (\lambda_1, \lambda_2, \dots)$ is a composition whose parts are weakly decreasing. The \emph{conjugate partition} $\lambda' = (\lambda'_1, \lambda'_2, \dots)$ is defined by $\lambda'_i = |\{j \mid \lambda_j \geq i\}|$.

\subsection{Quasisymmetric Functions}

A \emph{quasisymmetric function} in the variables $x_1, x_2, \dots$ (with coefficients in $\C$) is a formal power series $f(\mathbf{x}) \in \C[[\mathbf{x}]]$ of bounded degree such that, for any composition $\alpha$, the coefficient of $x_1^{\alpha_1}x_2^{\alpha_2}\cdots x_k^{\alpha_k}$ equals the coefficient of $x_{i_1}^{\alpha_1}x_{i_2}^{\alpha_2}\cdots x_{i_k}^{\alpha_k}$ whenever $i_1 < i_2 < \dots < i_k$.  We denote the algebra of quasisymmetric functions by $\QSym = \bigoplus_{n \geq 0} \QSym_n$, graded by degree.

There are two natural bases for $\QSym$, the monomial basis and the fundamental basis.  The \textit{monomial quasisymmetric function basis} $\{M_\alpha\}$, indexed by compositions $\alpha$, is given by 
\[M_\alpha = \sum_{1\leq i_1<i_2<\dots <i_k}  x_{i_1}^{\alpha_1}x_{i_2}^{\alpha_2}\cdots x_{i_k}^{\alpha_{k}}.\]
For example, $M_{(2, 1)} = \sum_{i<j}x_i^2x_j$.

The \textit{fundamental quasisymmetric function basis} $\{L_\alpha\}$ is also indexed by compositions $\alpha$ and is given by
\[L_\alpha = \sum_{\substack{i_1 \leq \dots \leq i_n\\i_s < i_{s+1} \text{ if } s \in D(\alpha)}} x_{i_1}x_{i_2} \cdots x_{i_n}.\]
In terms of the monomial basis,
\[L_\alpha = \sum_{\beta \preceq \alpha}M_\beta,\]
where the sum runs over all refinements $\beta$ of $\alpha$.  (A composition $\beta$ is a \emph{refinement} of $\alpha$ if $\alpha$ can be obtained by adding together adjacent parts of the composition $\beta$.)

%\medskip

For any labeled poset $P$ (on the ground set $[n]$), $K_{P}(\mathbf{x})$ is a quasisymmetric function, and we can express it in terms of the fundamental basis $\{L_\alpha\}$ using the linear extensions of $P$. For any linear extension $\pi \in \mathcal L(P)$, define the \emph{descent set} of $\pi$ to be $\des(\pi) = \{i \mid \pi(i) > \pi(i+1)\}$. We abbreviate $\co(\des(\pi))$ by $\co(\pi)$.

\begin{thm}[\cite{Stanley}]\label{L Expansion}
Let $P$ be a (labeled) poset on $[n]$.  Then
\[ K_{P}(\mathbf{x}) = \sum_{\pi \in \mathcal{L}(P)}L_{\co(\pi)}.\]
\end{thm}
In other words, the descent sets of the linear extensions of $P$ determine its $P$-partition generating function.  %This fact will be useful later when showing that two naturally labeled posets $P$ and $Q$ have different $P$-partition generating functions because we can just show that a linear extension of $P$ has a descent set that no linear extension of $Q$ has.

\subsection{Antichains and Shape}
An \textit{antichain} is a subset $A$ of a poset $P$ such that any two elements of $A$ are incomparable.
The antichain structure of a naturally labeled poset $P$ plays an important role in determining which sets can appear as descent sets for linear extensions of $P$.  Since $P$ is naturally labeled, elements $i < j$ form an antichain in $P$ if and only if there exists a linear extension of $P$ in which $j$ appears immediately before $i$.  This means that every descent in a linear extension of $P$ is formed by a $2$-element antichain.  Similarly, if there is a linear extension of $P$ that has $i$ consecutive descents, then these elements form an $(i+1)$-element antichain in $P$.  This shows that the sizes of some of the antichains of $P$ can be obtained from $K_P(\mathbf{x})$.

The following theorem of Greene \cite{GREENE} (that generalizes Dilworth's theorem \cite{Dilworth}) defines an important invariant related to the chains and antichains of a poset.

%For $k = 0, 1, \dots, $ let $a_k$ (resp.\ $c_k$) denote the maximum cardinality of a union of $k$ antichains (resp.\ chains) in $P$.  Let $\lambda_k = c_k - c_{k-1}$ and $\lambda'_k = a_k - a_{k-1}$ for all $k \geq 1$ so that
%\begin{align*}
%\lambda_1 + \cdots + \lambda_k &= c_k,\\
%\lambda_1' + \cdots + \lambda_k' &= a_k.
%\end{align*}
\begin{thm}[The Duality Theorem for Finite Partially Ordered Sets, \cite{GREENE}] \label{duality}  For any finite poset $P$, there exists a partition $\lambda = (\lambda_1, \lambda_2, \dots)$ such that $\lambda_1 + \lambda_2 + \cdots + \lambda_k$ is the maximum cardinality of the union of $k$ chains in $P$ for all $k \geq 1$.
	
Moreover, if $\lambda' = (\lambda_1', \lambda_2', \dots)$ is the conjugate partition to $\lambda$, then $\lambda_1' + \lambda_2' + \cdots + \lambda_k'$ is the maximum cardinality of the union of $k$ antichains in $P$ for all $k \geq 1$.
\end{thm}

The \emph{shape} of a finite poset $P$ is therefore defined to be the partition $\lambda$ that satisfies Theorem~\ref{duality}.
%
%\begin{defn} The \emph{shape} of a finite poset, denoted $\sh(P)$, is the partition $\lambda$ that satisfies $\lambda'_1 + \dots + \lambda'_k = a_k$ for all $k$, where $a_k$ is the largest number of elements in a union of $k$ antichains of $P$.
%\end{defn}

The \textit{width} of a poset $P$ is the length of its longest antichain.  If $\sh(P) = \lambda$, then the width of $P$ is $\lambda_1'$.

\subsection{Poset of Order Ideals}
An \emph{order ideal}, or \emph{ideal} for short, is a subset $I \subseteq P$ such that if $x \in I$ and $y \prec x$, then $y \in I$.  There is a one-to-one correspondence between ideals and antichains, namely, the maximal elements of an ideal form an antichain.  A \emph{principal order ideal} is an ideal with a unique maximal element.  The dual notion of an order ideal is a \emph{filter}: it is a subset $J \subseteq P$ such that if $x \in J$ and $y \succ x$, then $y \in J$.

The set of all order ideals of $P$, ordered by inclusion, forms a poset that we will denote $J(P)$.  In fact, $J(P)$ is a finite (graded) distributive lattice. The rank of an element of $J(P)$ is the number of elements in the corresponding ideal of $P$.

%If $I$ is an ideal, then the number of elements $I$ covers is equal to the number of maximal elements of $I$.  Similarly, the number of elements that cover $I$ in $J(P)$ is equal to the number of minimal elements of $P\setminus I$.
%Let $\antikij$ be the function that sends a poset $P$ to the number of $k$-element ideals $I$ of $P$ such that $I$ has $i$ maximal elements and $P \setminus I$ has $j$ minimal elements.  This is equal to the number of rank $k$ elements of $J(P)$ that cover $i$ elements and are covered by $j$ elements.  Observe that 
%\[\sum_{i, j} \antikij(P) = \rank_k(P),\]
%where $\rank_k(P)$ is number of $k$-element ideals of $P$.

If $J(P)$ has a unique element of some rank $k$, then $P$ can be expressed as an \emph{ordinal sum} $P = Q \oplus R$ with $|Q|=k$.  Here, the ordinal sum $Q \oplus R$ is the poset on the disjoint union $Q \sqcup R$ with relations $x\preceq y$ if and only if $x \preceq_Q y$, $x \preceq_R y$, or $x \in Q$ and $y \in R$.

\begin{defn}
A finite poset $P$ is \emph{irreducible} if $P = Q \oplus R$ implies that either $Q=\varnothing$ or $R=\varnothing$.
\end{defn}

Each poset has a unique \emph{ordinal sum decomposition} $P = P_1 \oplus P_2 \oplus \dots \oplus P_k$ with $P_i$ irreducible. If $|P_i|=n_i$, then $J(P)$ has exactly one element in ranks $0$, $n_1$, $n_1 + n_2$, \dots,  $n_1 + n_2 + \dots + n_k$.

\begin{lemma}\label{K of reducible}
Suppose $P$ and $Q$ have ordinal sum decompositions $P = P_1 \oplus P_2 \oplus \dots \oplus P_k$ and $Q = Q_1 \oplus Q_2 \oplus \dots \oplus Q_j$.  If $\KP = \KQ$, then $k = j$, and $K_{P_i}(\mathbf{x}) = K_{Q_i}(\mathbf{x})$ for $i = 1, \dots, k$.
\end{lemma}

\begin{proof}

%Old version
%Since $P$ is naturally labeled, if elements $a$ and $b$ form a descent in a linear extension of $P$, then $a$ and $b$ must both lie in the same $P_i$ for some $i$.  This means that no linear extension of $P$ has a descent in the locations $n_1$, $n_1 + n_2$, \dots, $n_1+n_2+ \dots + n_{k-1}$. For any other possible descent location $r$, let $I$ be an ideal of $P$ of size $r$ containing an element whose label $x$ is as large as possible. Since $I$ is not an ordinal summand of $P$, some minimal element $y$ of $P \setminus I$ is not greater than $x$ in $P$. Then $(I \cup \{y\}) \setminus \{x\}$ is also an ideal of $P$ of size $r$, so we must have $y < x$ by our choice of $I$. Hence there is a linear extension of $P$ with a descent in location $r$ (that begins with the elements of $I$ ending with $x$, followed by $y$). Since we can determine all possible descent sets of linear extensions of $P$ from the expansion of $\KP$ in the fundamental basis by Theorem~\ref{L Expansion}, we can thereby determine $k$ and all $n_i=|P_i|$ from $\KP$.

%New Version
We first show that if $n_i = |P_i|$, then 
\[ \bigcup_{\pi \in \mathcal{L}(P)} \des(\pi) = [n-1]\setminus \{n_1, n_1 + n_2, \dots, n_1+n_2+ \dots + n_{k-1}\}.\]

Since $P$ is naturally labeled, if elements $a$ and $b$ form a descent in a linear extension of $P$, then $a$ and $b$ must both lie in the same $P_i$ for some $i$.  This means that no linear extension of $P$ has a descent in the locations $n_1$, $n_1 + n_2$, \dots, $n_1+n_2+ \dots + n_{k-1}$.
For any other $r \in [n-1]$,
let $I$ be an ideal of $P$ of size $r$ containing an element whose label $x$ is as large as possible. Since $I$ is not an ordinal summand of $P$, some minimal element $y$ of $P \setminus I$ is not greater than $x$ in $P$. Then $(I \cup \{y\}) \setminus \{x\}$ is also an ideal of $P$ of size $r$, so we must have $y < x$ by our choice of $I$. Hence there is a linear extension of $P$ with a descent in location $r$ (that begins with the elements of $I$ ending with $x$, followed by $y$).

Since we can determine all possible descent sets of linear extensions of $P$ from the expansion of $\KP$ in the fundamental basis by Theorem~\ref{L Expansion}, we can thereby determine $k$ and all $n_i=|P_i|$ from $\KP$.

To get $K_{P_i}(\mathbf{x})$ from $\KP$, note that linear extensions of $P$ can be broken up into $k$ parts: the first $n_1$ elements form a linear extension of $P_1$, the next $n_2$ elements form a linear extension of $P_2$, and so on. Then define $\pi_i \colon \QSym_n \rightarrow \QSym_{n_i}$ by
\[
\pi_i(L_{\alpha}) =
\begin{cases}
L_{\beta} &\text{if } \alpha = (n_1 + \dots + n_{i-1})\odot \beta \odot (n_{i+1}+ \dots +n_k), \\
0 &\text{otherwise.}
\end{cases}
\]
extended linearly. It follows that $K_{P_i}(\mathbf{x}) = \pi_i(\KP)$.
\end{proof}

For any $S \subseteq [n]$, define the subposet $J(P)_S = \{I \in J(P) \mid |I| \in S\}$.  Let $f_S$ denote the number of maximal chains in $J(P)_S$.  The function $f\colon 2^{[n]} \rightarrow \mathbb{Z}$ is called the \textit{flag $f$-vector} of $J(P)$.  Also define $h_S$ by
\[ h_S = \sum_{T \subseteq S} (-1)^{|S-T|}f_T. \]
This function $h$ is called the \textit{flag $h$-vector} of $J(P)$.

In the case when $P$ is naturally labeled, the flag $f$-vector $f_S$ and the flag $h$-vector $h_S$ of $J(P)$ appear in the expansion of $K_P(\mathbf{x})$ as follows:
\[K_P(\mathbf{x}) = \sum_{\alpha}f_{D(\alpha)}M_\alpha = \sum_{\alpha}h_{D(\alpha)} L_\alpha.\]
%Here $f_S$ and $h_S$ are the flag $f$-vector and flag $h$-vector of $J(P)$, respectively.

%The \textit{M-support of $\KP$} is defined by 
%\[\supp_M(\KP) = \{ \alpha \mid f_{D(\alpha)} \neq 0\}.\]
The \textit{L-support of $\KP$} is defined by 
\[\supp_L(\KP) = \{ \alpha \mid h_{D(\alpha)} \neq 0\}.\]

\subsection{Hopf Algebra}
Let $\mathcal{J}$ denote the set of all finite distributive lattices up to isomorphism. The free $\C$-module $\C[\mathcal{J}]$, whose basis consists of isomorphism classes of distributive lattices $[J] \in \mathcal{J}$, can be given a Hopf algebra structure known as the \textit{reduced incidence Hopf algebra}. The multiplication, unit, comultiplication, and counit are defined as follows:
\begin{align*}
\nabla([J_1] \otimes [J_2])&:= [J_1 \times J_2],\\
1_{\C[\mathcal{J}]} &:= [o],\\
\Delta([J]) &:= \sum_{x \in J} [\hat{0}, x] \otimes [x, \hat{1}],\\
\epsilon([J]) &:= 
\begin{cases}
            1 & \text{if } |J|=1, \\
            0 & \text{otherwise.}
\end{cases}
\end{align*}
Here $[o]$ is the isomorphism class of the one-element lattice, and $\hat 0$ and $\hat 1$ are the minimum and maximum elements of a lattice.

In fact, the reduced incidence Hopf algebra can be made into a combinatorial Hopf algebra after choosing an appropriate character.  A \emph{combinatorial Hopf algebra} $\mathcal{H}$ is a graded connected Hopf algebra over a field $\C$ equipped with a character (multiplicative linear function) $\zeta \colon \mathcal{H} \rightarrow \C$ (see  \cite{ABS} for more details).  We define the character of the reduced incidence Hopf algebra to be the map $\zeta \colon \C[\mathcal{J}] \rightarrow \C$ defined on basis elements by $\zeta([J]) = 1$ for all $J$ and extended linearly.

These functions can likewise be defined on the free $\C$-module $\C[\mathcal{P}]$ whose basis consists of isomorphism classes of finite posets.  Explicitly:
\begin{align*}
\nabla([P_1] \otimes [P_2])&:= [P_1 \sqcup P_2],\\
1_{\C[\mathcal{P}]} &:= \varnothing,\\
\Delta([P]) &:= \sum_{\text{ideal } I \subseteq P} [I] \otimes [P \setminus I],\\
\epsilon([P]) &:= \begin{cases}
            1 & \text{if } |P|=0, \\
            0 & \text{otherwise.}
        \end{cases}
\end{align*}

The corresponding character of $\C[\mathcal{P}]$ is $\zeta_{\mathcal{P}}\colon  \C[\mathcal{P}] \rightarrow \C$ defined by $\zeta([P]) = 1$ for all $P$, extended linearly.
These functions are all compatible with the map $J$ that sends $[P]$ to $[J(P)]$, so $J$ is a Hopf isomorphism between $\C[\mathcal P]$ and $\C[\mathcal J]$. (For more information on these Hopf algebras, see \cite{Schmitt}.)

We define the \textit{graded comultiplication} $\Delta_{k, n-k}([P])$ to be the part of $\Delta([P])$ of bidegree $(k, n-k)$, that is,
\[\Delta_{k, n-k}([P]) := \sum_{\substack{I \subseteq P \\ |I| = k}} [I] \otimes [P \setminus I].\]

The ring of quasisymmetric functions $\QSym$ is also a Hopf algebra.  The comultiplication is defined on the fundamental quasisymmetric function basis by
\[\Delta(L_{\alpha}) := \sum_{\substack{(\beta, \gamma) \\ \alpha = \beta \cdot \gamma \text{ or } \beta \odot \gamma}} L_{\beta} \otimes L_{\gamma}.\]
The \textit{graded comultiplication} $\Delta_{k, n-k}(L_{\alpha})$ is given  by
\[\Delta_{k, n-k}(L_{\alpha}) := \sum_{\substack{(\beta, \gamma) \\\alpha = \beta \cdot \gamma \text{ or } \beta \odot \gamma \\ |\beta| = k}} L_{\beta} \otimes L_{\gamma}.\]

The map $K\colon \C[\mathcal{P}] \rightarrow \QSym$ that sends $P$ to the $P$-partition generating function $K_P(\mathbf{x})$ is the unique Hopf morphism that satisfies $\zeta_{\mathcal{P}} = \zeta_{\mathcal{Q}} \circ K$, where the character $\zeta_{\mathcal Q}$ for $\QSym$ is the linear function that sends $L_{(n)}$ to $1$ for all $n$ and all other $L_{\alpha}$ to $0$.

%%%%%%%%%%%%%%%%%%NECESSARY CONDITIONS%%%%%%%%%%%%%%%%%%%%

\section{Necessary Conditions}
In this section, we will describe various necessary conditions for two naturally labeled posets to have the same partition generating function.

\subsection{Order ideals and antichains}
Let $\antikij$ be the function that sends a poset $P$ to the number of $k$-element ideals $I$ of $P$ such that $I$ has $i$ maximal elements and $P \setminus I$ has $j$ minimal elements.  This is equal to the number of rank $k$ elements of $J(P)$ that cover $i$ elements and are covered by $j$ elements. We will show that if $\KP = \KQ$, then $\antikij(P) = \antikij(Q)$ for all $k$, $i$, and $j$.

First we will need the following lemmas.

\begin{lemma}\label{maxcount}
Let $P$ be a naturally labeled finite poset.
\begin{enumerate}[(a)]
\item If $P$ has exactly $j$ maximal elements, then there are $\binom{j-1}{k}$ linear extensions of $P$ whose descent set is $\{n-k, n-k+1, \dots, n-1\}$.
\item If $P$ has exactly $j$ minimal elements, then there are $\binom{j-1}{k}$ linear extensions of $P$ whose descent set is $\{1, 2, \dots , k\}$.
\end{enumerate}
\end{lemma}

\begin{proof}
Let $\sigma = \sigma_1\sigma_2\cdots\sigma_n$ be a linear extension of $P$, and suppose $\des(\sigma) = \{n-k, n-k+1, \dots, n-1\}$.  It follows that 
\[\sigma_1<\sigma_2< \dots < \sigma_{n-k} > \sigma_{n-k+1} > \dots > \sigma_n.\]
This implies that $\sigma_{n-k} = n$ and $\{\sigma_{n-k+1}, \dots, \sigma_n\}$ must be maximal elements,  for if $\sigma_i$ is not maximal for $i>n-k$, then there is some $\sigma_j \succ \sigma_i$ in $P$ (with $j > i>n-k$ since $\sigma$ is a linear extension). But since $P$ is naturally labeled, this would imply $\sigma_i < \sigma_j$, which is impossible.  Therefore  $\{\sigma_{n-k+1}, \dots, \sigma_n\}$ is a $k$-element subset of the maximal elements of $P$ other than $n$.  There are $\binom{j-1}{k}$ such subsets, and each corresponds to a linear extension with the desired descent set.
%Since $P$ is naturally labeled we can choose to label the $j$ maximal elements $n, n-1, \dots, n-j+1$.  The set of maximal elements of $P$ is an antichain in $P$.  Let $a =a_1a_2\dots a_n$ be a linear extension of $P$ with descent set $\{n-k, n-k+1, \dots, n-1\}$.  So we must have that $a_1 < a_2 < \dots < a_{n-k} > a_{n-k+1} > \dots > a_n$, implying that $a_{n-k} = n$ and $\{a_{n-k}, a_{n-k+1}, \dots, a_n\}$ must be a $(k+1)$-element antichain containing $n$ since it appears in decreasing order in $a$.  There are $\binom{j-1}{k}$ $(k+1)$-element antichain containing $n$ made up of only maximal elements.

%Suppose there is a linear extension of $P$ with the desired descent set whose decreasing subsequence has an entry, $b$, that isn't maximal.  $b$ must appear at the end of the linear extension so $b$ must be unrelated to all entries greater than $b$.  So $b$ is unrelated to all of the maximal elements making $b$ maximal.  This is a contradiction.  Therefore there are only $\binom{j-1}{k}$ linear extensions of $P$ with descent set is $\{n-k, n-k+1, \dots, n-1\}$.

The proof for (b) follows similarly.
\end{proof}

\begin{lemma}\label{maxmin}
\begin{enumerate}[(a)]
\item
There exists a linear function $\max_i\colon \QSym \rightarrow \mathbb{C}$ satisfying
\[
\max_i\nolimits(K_P(\mathbf{x})) = \begin{cases}
            1 & \text{if $P$ has exactly i maximal elements,} \\
            0 & \text{otherwise.}
        \end{cases}
\]
\item
There exists a linear function $\min_i\colon \QSym \rightarrow \mathbb{C}$ satisfying
\[
  \min_i\nolimits(K_P(\mathbf{x})) =
        \begin{cases}
            1 & \text{if $P$ has exactly i minimal elements,} \\
            0 & \text{otherwise.}
        \end{cases}
\]
\end{enumerate}
\end{lemma}

\begin{proof}
We claim that the following function defined on the basis $\{L_\alpha\}$ of $\QSym_n$, extended linearly, satisfies this condition:
\[
\max_i\nolimits(L_{\alpha}) = 
	\begin{cases}
            (-1)^{(k-i+1)}\binom{k}{i-1} &\text{if } \alpha = \alpha(k) := (n-k-1)\odot 1^{k+1} \text{ for } i-1 \leq k < n, \\
            0 &\text{otherwise.}
        \end{cases}
\]
By Theorem \ref{L Expansion}, 
$K_P(\mathbf{x}) = \sum_{\alpha}c_{\alpha}L_{\alpha}$, where $c_\alpha$ is the number of linear extensions of $P$ with descent set $D(\alpha)$.  Evaluating $\max_i$ on $K_P(\mathbf{x})$, we have

\[\max_i\nolimits(K_P(\mathbf{x})) = \max_i\nolimits\left(\sum_{\alpha}c_{\alpha}L_{\alpha}\right) = \sum_{\alpha}c_{\alpha}\max_i\nolimits(L_{\alpha}) = \sum_{k=0}^n c_{\alpha(k)} (-1)^{(k-i+1)}\binom{k}{i-1}.\]

Suppose $P$ has exactly $j$ maximal elements.  By Lemma \ref{maxcount}, $c_{\alpha(k)} = \binom{j-1}{k}$ because $D(\alpha(k)) = \{n-k, n-k+1, \dots, n-1\}$.  Substituting this equality into the summation we have:
%\[\max_i\nolimits(K_P(\mathbf{x})) = \sum_{k=0}^n (-1)^{(k-i+1)}\binom{j-1}{k}\binom{k}{i-1}.\]
%For completeness we will prove the following combinatorial identity:  
%\[\sum_{k=0}^n (-1)^{(k-i+1)}\binom{j-1}{k}\binom{k}{i-1} = \delta_{i, j}.\]
\begin{align*}
\max_i\nolimits(\KP) & = \sum_{k=0}^n (-1)^{(k-i+1)}\binom{j-1}{k}\binom{k}{i-1} \\
&= \binom{j-1}{i-1}\sum_{k=0}^n (-1)^{(k-i+1)}\binom{j-i}{k-i+1} \\
& =\binom{j-1}{i-1} \delta_{i, j}\\
& =\delta_{i, j}.
\end{align*}

The proof of (b) follows similarly.
\end{proof}

It follows from Lemma~\ref{maxmin} that $\max_i(\KP)$ and $\min_i(\KP)$ can be expressed as a linear combination of the coefficients of the fundamental basis expansion of $\KP$.  Observe that if we order the compositions in lexicographic order, then the leading term in the expansion of $\max_i(\KP)$ is $(n-i)\odot 1^i$, and the leading term in the expansion of $\min_i(\KP)$ is $1^i \odot (n-i)$.  We will now use these linear functions along with the coproduct to express $\antikij(P)$ in terms of $\KP$.

\begin{thm}\label{Anti}
If $K_P(\mathbf{x}) = K_Q(\mathbf{x})$, then $\antikij(P) = \antikij(Q)$ for all triples $(k, i, j)$.
\end{thm}

\begin{proof}
We will prove this result by finding a linear function that takes $\KP$ to $\antikij(P)$.

Recall that there is a Hopf morphism $K\colon \mathcal{P} \rightarrow \QSym$ that sends $P$ to $\KP$. %together with $\zeta_{\mathcal{P}}$ induces a
%In fact, $K$ is the unique graded Hopf morphism such that $\zeta_{\mathcal{P}} = \zeta_{\mathcal{Q}} \circ K$ (see e.g. Theorem 7.3 in \cite{GrinbergReiner}).
It follows that $K$ is compatible with comultiplication, $(K \otimes K) \circ \Delta = \Delta \circ K$, and graded comultiplication, $(K \otimes K) \circ \Delta_{k, n-k} = \Delta_{k, n-k} \circ K$.%  This morphism is the $P$-partition generating function for P.

Define $\max_i\nolimits^*\colon \C[\mathcal{P}] \rightarrow \mathbb{\C}$ by $\max_i\nolimits^* = \max_i\nolimits \circ K$.  Thus $\max_i\nolimits^*(P) = 1$ if $P$ has exactly $i$ maximal elements, otherwise $\max_i\nolimits^*(P) = 0$.  Similarly define $\min_i\nolimits^* = \min_i\nolimits \circ K$.

Consider the following commutative diagram:

\begin{center}
\begin{tikzpicture}
  \matrix (m) [matrix of math nodes,row sep=3em,column sep=4em,minimum width=2em]
  {
     P & \sum P_{(1)} \otimes P_{(2)}&& \\
     K_P(\mathbf{x}) & \sum K_{P_{(1)}}(\mathbf{x}) \otimes K_{P_{(2)}}(\mathbf{x}) && \antikij(P) \\
     };
  \path[-stealth]
    (m-1-1) edge (m-2-1)
            edge node [above] {$\Delta_{k, n-k}$}(m-1-2)
    (m-1-1) edge node [left] {$K$} (m-2-1)
    (m-2-1.east|-m-2-2) edge 
            node [above] {$\Delta_{k, n-k}$} (m-2-2)
    (m-1-2) edge node [left] {$K \otimes K$} (m-2-2)
    (m-2-2) edge (m-2-4)
    	  edge node [below] {$\max_i \otimes \min_j$} (m-2-4)
    (m-1-2) edge (m-2-4)
    	 edge node [above right] {$\max_i^*\nolimits \otimes \min_j^*\nolimits$} (m-2-4)
    ;
\end{tikzpicture}
\end{center}

We can compute $\antikij(P)$ by evaluating the composition of the top row of functions on $P$ as
\[\antikij(P) = ((\max_i^*\nolimits \otimes \min_j^*\nolimits) \circ \Delta_{k, n-k})(P),\]
or equivalently we can compute $\antikij(P)$ by evaluating the composition of the bottom row of functions to $K_P(\mathbf{x})$ as
\[\antikij(P) = ((\max_i\nolimits \otimes \min_j\nolimits) \circ \Delta_{k, n-k})(K_P(\mathbf{x})).\]
This shows that $\antikij(P)$ only depends on $K_P(\mathbf{x})$. Therefore if two posets $P$ and $Q$ have the same partition generating function, then $\antikij(P) =\antikij(Q)$.
\end{proof}

In particular, by summing over $k$ and $j$, we arrive at the following corollary, conjectured by McNamara and Ward \cite{McNamaraWard}.

\begin{cor}
	If $\KP=\KQ$, then $P$ and $Q$ have the same number of antichains of each size.
\end{cor}

We have just shown that $\antikij(P)$ is a linear function of $\KP$, so in particular, it can be expressed as a linear combination of certain coefficients of the fundamental basis expansion of $\KP$. In fact, $((\max_i \otimes \min_j) \circ \Delta_{k, n-k})(L_\alpha) = 0$ unless $\alpha$ is of the form $\alpha = (a) \odot 1^b \odot 1^c \odot (n-a-b-c)$, so $\antikij(P)$ only depends on the coefficients for these compositions in $\KP$. If we order the compositions in lexicographic order, then the leading coefficient of $\antikij(P)$ is $c_{\alpha(k, i, j)}(P)$, where
\[\alpha(k, i, j) = (k-i) \odot 1^i \odot 1^j \odot (n-k-j).\] One can then deduce the following result.
\begin{cor}

Let $c_\alpha(P)$ and $c_\alpha(Q)$ denote the coefficent of $L_\alpha$ in $\KP$ and $K_Q(\mathbf{x})$, respectively.  If $\antikij(P) = \antikij(Q)$ for all $k, i,j$, then $c_\alpha(P) = c_\alpha(Q)$ for all compositions $\alpha$ of the form $\alpha = (a) \odot 1^b \odot 1^c \odot (n-a-b-c)$.

\end{cor}

\begin{proof}

Let $C = \{\alpha \mid \alpha = (a) \odot 1^b \odot 1^c \odot (n-a-b-c)\}$.  Define $\beta(a, b, c)$ by $\beta(a, b, c) = (a) \odot 1^b \odot 1^c \odot (n-a-b-c)$. We showed in Theorem \ref{Anti} that for all triples $(k, i, j)$, $\antikij(P)$ can be expressed as a linear combination of $c_\alpha$ for $\alpha \in C$.  Each of these $c_\alpha$ appear as the leading coefficient in the expansion of some $\antikij(P)$.  In particular, $c_{\beta(a, b, c)}$ is the leading coefficient of the expansion for $\anti_{a+b, b, c}(P)$.  Therefore the matrix that expresses $\antikij(P)$ as a linear combination of $c_{\beta(a,b,c)}$ has full rank, so the coefficient of $L_\alpha$ in $\KP$ is determined by the values of $\antikij(P)$ for all $\alpha \in C$.
\end{proof}

This shows that some easily counted statistics on $J(P)$ determine a number of the coefficients in the fundamental basis expansion of $\KP$.

As another simple application, we can apply $\max_i$ along with the coproduct to see the following result, which will be used later as a tool in showing that two posets do not have the same partition generating function.  %A similar result can be stated when there is a unique filter $I$ such that $|I|=k$ and $I$ has $i$ minimal elements.

\begin{cor}\label{P without I}
	Suppose that for some $k$ and $i$, $P$ has a unique ideal $I$ of size $k$ with $i$ maximal elements. Then $K_{P\backslash I}(\mathbf{x})$ can be determined from $\KP$.
\end{cor}

\begin{proof}  The partition generating function for $P\backslash I$ is
	\[K_{P\backslash I}(\mathbf{x}) = (\max_i\nolimits \otimes id) \Delta_{k, n-k} \KP. \qedhere\]
\end{proof}

\subsection{Jump}
Let the \textit{jump} of an element be the maximum number of relations in a saturated chain from the element down to a minimal element.  We define the \textit{jump sequence} to be $\jump(P) = (j_0, \dots, j_k)$, where $j_i$ equals the number of elements with jump $i$, and $k$ is the maximum jump of an element.

McNamara and Ward prove in \cite{McNamaraWard} that if two posets have the same $P$-partition generating function, then they must have the same jump sequence.  The jump sequence of a naturally labeled poset can be interpreted in terms of minimal elements.  Let $P_i$ denote the subposet of $P$ that consists of elements of $P$ with jump greater than or equal to $i$. Then $j_i$ is equal to the number of minimal elements of $P_i$, and $P_{i+1}$ is obtained from $P_{i}$ by removing its minimal elements. McNamara and Ward prove the following result.
\begin{lemma}[\cite{McNamaraWard}, Corollary 5.3]\label{Remove Jump}
	If $P$ and $Q$ have the same partition generating function, then so do $P_i$ and $Q_i$, the induced subposets consisting of elements of jump at least $i$.
\end{lemma}

We prove a similar result on the $L$-support of $\KP$.
\begin{lemma}\label{jump at least i}  If $\KP$ and $\KQ$ have the same $L$-support, then so do $K_{P_i}(\mathbf x)$ and $K_{Q_i}(\mathbf x)$, the partition generating functions for the induced subposets consisting of elements of jump at least $i$.
\end{lemma}

\begin{proof} 
	Any linear extension of $P$ that begins with $j_0-1$ descents must start with the minimal elements of $P$ in descending order followed by a linear extension of $P_1$, and no linear extension can start with more descents. Thus $\alpha \in \supp_L(K_{P_1}(\mathbf x))$ if and only if $1^{j_0}\odot \alpha \in \supp_L(\KP)$, where $j_0$ is the maximum value for which some such $\alpha$ exists. We can repeat this $i$ times to see that  $\beta \in \supp_L(K_{P_i}(\mathbf{x}))$ if and only if $(1^{j_0}\odot 1^{j_1} \odot \dots \odot 1^{j_{i-1}} \odot \beta) \in \supp_L(K_{P}(\mathbf{x}))$.
\end{proof}

A similar proof can be used to give an alternate argument for Lemma~\ref{Remove Jump}.

We define the \textit{upward jump} of an element to be the maximum number of relations in a saturated chain from the element up to a maximal element.  We define the \textit{upward jump sequence} to be $\upjump(P) = (j'_0, \dots, j'_k)$, where $j'_i$ equals the number of elements with upward jump $i$, and $k$ is the maximum up-jump of an element.  We then let the \textit{jump pair} of an element $x$ be $\jumppair(x) = (\jump(x), \upjump(x))$.

\begin{lemma}\label{jumppair}  If $\supp_L(\KP) = \supp_L(\KQ)$, then $P$ and $Q$ have the same number of elements with jump pair $(i, j)$ for all $i$ and $j$.
\end{lemma}
 
 \begin{proof}
 Let $P_{i,j}$ be the induced subposet of $P$ consisting of all elements with jump at least $i$ and up-jump at least $j$.  By the previous lemma and its dual, $\supp_L(P_{i,j})$ is determined by $\KP$, hence so is $|P_{i,j}|$.  This implies the result since the number of elements with jump pair $(i, j)$ is $|P_{i,j}| - |P_{i+1, j}| - |P_{i,j+1}| + |P_{i+1,j+1}|$ by inclusion-exclusion.
 \end{proof}

Another similar result is given in the following lemma.

\begin{lemma}\label{jumpideal} If $\KP = \KQ$, then $P$ and $Q$ have the same number of elements with principal order ideal size $i$ and up-jump $j$.
\end{lemma}

\begin{proof} Let $P_{0,j}$ be the induced subposet consisting of elements with up-jump at least $j$.  From the dual of Lemma~\ref{jump at least i}, the generating function for $P_{0,j}$ is determined by $\KP$.  The number of elements with principal order ideal size $i$ and up-jump $j$ in $P$ is the same as the number of maximal elements with principal order ideal size $i$ in $P_{0,j}$.  The function $(\max_1 \otimes \zeta)\Delta_{i, |P_{0,j}|-i}$ evaluated on $K_{P_{0,j}}(\mathbf{x})$ gives us the number of elements in $P_{0,j}$ whose principal order ideal has $i$ elements.  We can count the number of these that are maximal by evaluating $(\max_1 \otimes \zeta)\Delta_{i, |P_{0,j+1}|-i}$ on $K_{P_{0,j+1}}(\mathbf{x})$ and taking the difference.
\end{proof}

\subsection{Shape}
Next, we show that the shape of the poset $P$ is determined by $\KP$, or more specifically, by its support.
\begin{thm} \label{shape}
If $\supp_L(\KP) = \supp_L(\KQ)$, then $\sh(P) = \sh(Q)$.
\end{thm}

\begin{proof}
Let $\alpha = (\alpha_1, \alpha_2, \dots , \alpha_k)$ be a composition of $n$, and let $\sh(P) = \lambda$.  Define $B(\alpha) = \#\{a \mid a\in D(\alpha)\text{ and }a-1 \notin D(\alpha)\}$, that is, $B(\alpha)$ is the number of decreasing runs (with at least two elements) in a permutation with descent set $D(\alpha)$.  We also define $L_i(\alpha) = i + |D(\alpha)|$.

We will prove that the shape of $P$ is determined by its support by showing that, for $i \leq \lambda_1$,
\[\lambda'_1 + \dots + \lambda'_i = \max\{L_i(\alpha) \mid \alpha \in \supp_L(\KP) \text{ and } B(\alpha) \leq i\}.\]
First, choose some $\alpha$ appearing on the right hand side. There is a linear extension of $P$ with descent set $D(\alpha)$ that has at most $i$ decreasing runs. These decreasing runs (together with possibly some single elements) correspond to $i$ antichains of $P$, and the total number of elements in the union of these antichains is $L_i(\alpha)$. Since $\lambda'_1 + \dots + \lambda'_i$ is by definition the largest number of elements in a union of $i$ antichains of $P$, \[\max\{L_i(\alpha) \mid \alpha \in \supp_L(\KP) \text{ and } B(\alpha) \leq i\} \leq \lambda_1' + \dots + \lambda_i'.\]

Conversely, let $A_1$, $A_2$, \dots, $A_i$ be antichains such that $|A_1| + |A_2| + \dots + |A_i| = \lambda'_1 + \dots + \lambda'_i$. Without loss of generality, we can take $A_j \leq A_{j+1}$ for $j =1, \dots i-1$, meaning that for each $y \in A_{j+1}$, there exists $x \in A_{j}$ with $x \preceq y$.  We can do this because in the subposet $A_1 \cup A_2 \cup \cdots \cup A_i$, the longest chain has at most $i$ elements, so we can redefine $A_1$ to be the elements with jump 0 in this subposet, $A_2$ to be the elements with jump 1, and so on.  %Since $P$ is naturally labeled, we can assume that the highest label in $A_j$ is less than the lowest label in $A_{j+1}$ for $j =1, \dots i-1$.  

For $j=1, \dots, i$, let $I_j$ denote the smallest order ideal containing $A_1 \cup \dots \cup A_j$. Then let $B_1 = I_1 \setminus A_1$, $B_j = I_j \backslash (A_j \cup I_{j-1})$ for $j = 2, \dots i$, and let $B_{i+1} = P \setminus I_i$.  %Since $P$ is naturally labeled, we can assume that for all $j$ the largest label of $B_j$ is less that the largest label in $A_j$ and lowest label of $B_j$ is greater than the lowest label of $A_{j-1}$.
There is a linear extension $\pi$ of $P$ of the form $\pi = B_1A_1B_2A_2\dots A_iB_{i+1}$, where the entries in each $B_j$ appear in increasing order and the entries in each $A_j$ appear in decreasing order.  It follows that
\[ L_i(\co(\pi)) - i = |\des(\pi)| \geq \sum_{j=1}^i (|A_j|-1) = \lambda'_1 + \dots + \lambda'_i - i. \]
Therefore $L_i(\co(\pi)) \geq \lambda'_1 + \dots + \lambda'_i$, which implies that
\[\lambda'_1 + \dots + \lambda'_i \leq \max\{L_i(\alpha) \mid \alpha \in \supp_L(\KP) \text{ and } B(\alpha) \leq i\}. \qedhere\]
\end{proof}

Therefore the shape of a poset $P$ is determined by the compositions that appear with a nonzero coefficient in the fundamental quasisymmetric function expansion of $\KP$.
\begin{cor}
If $K_P(\mathbf{x}) = K_Q(\mathbf{x})$, then $\sh(P) = \sh(Q)$.
\end{cor}
\begin{proof}  This result follows directly from the previous theorem.
\end{proof}

%%%%%%%%%%%%%%%%%%%UNIQUENESS FROM SHAPE%%%%%%%%%%%%%%%%%%%

\section{Uniqueness from shape}
Since Theorem \ref{shape} shows that posets with the same generating function must have the same shape, one can ask for which shapes is a poset of that shape uniquely determined by its generating function. In other words, for which $\lambda$ do all nonisomorphic posets of shape $\lambda$ have distinct partition generating functions?

We will prove that this holds for three cases below: width two posets, hook shaped posets, and nearly hook shaped posets.

\subsection{Width two posets}
In this section we consider posets whose shape has at most two parts, that is to say, the width of the poset is at most two.  Dilworth's theorem \cite{Dilworth} states that if the width of $P$ is $2$, then $P$ can be partitioned into $2$ chains, $C_1$ and $C_2$.  We will use the notation
\[P = C_1 \uplus C_2\]
to denote our choice of partition.  In the case when $P$ is irreducible, the minimal elements of $C_1$ and $C_2$ are the minimal elements of $P$.  We can embed $J(P)$ into $\mathbb{N}^2$ by mapping an ideal $I$ to the point $(a_1, a_2)$ where $a_i = |I \cap C_i|$.  Hence when referring to $J(P)$ we will treat it as a sublattice of $\mathbb{N}^2$.
\begin{example} \label{jp}
The following is a width 2 poset along with its poset of order ideals embedded in $\mathbb{N}^2$.
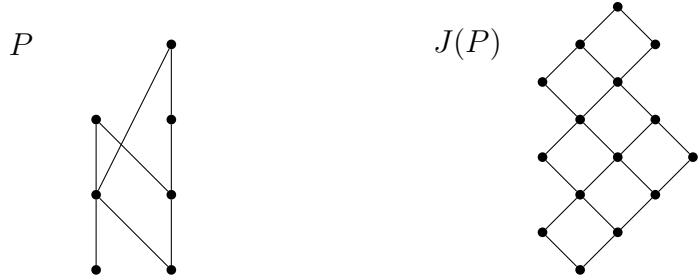
\begin{center}

\begin{tikzpicture}
 [auto,
 vertex/.style={circle,draw=black!100,fill=black!100,thick,inner sep=0pt,minimum size=1mm}]
\node(P) at (-1, 2) {$P$};
\node (v) at ( 0,-1) [vertex] {};
\node (v0) at ( 1,-1) [vertex] {};
\node (v1) at ( 0,0) [vertex] {};
\node (v2) at ( 1,0) [vertex] {};
\node (v4) at ( 0,1) [vertex] {};
\node (v5) at ( 1,1) [vertex] {};
\node (v7) at ( 1,2) [vertex] {};
\draw [-] (v1) to (v4);
\draw [-] (v2) to (v4);
\draw [-] (v2) to (v5);
\draw [-] (v5) to (v7);
\draw [-] (v1) to (v7);
\draw [-] (v) to (v1);
\draw [-] (v0) to (v2);
\draw [-] (v0) to (v1);
\end{tikzpicture}
\hspace{3cm}
\begin{tikzpicture}
 [auto,
 vertex/.style={circle,draw=black!100,fill=black!100,thick,inner sep=0pt,minimum size=1mm}, scale = .5]
\node(JP) at (-3, 4) {$J(P)$};
\node (a) at ( 0,-2) [vertex] {};
\node (b) at ( -1,-1) [vertex] {};
\node (c) at ( 1,-1) [vertex] {};
\node (d) at ( 2,0) [vertex] {};
\node (e) at ( 3,1) [vertex] {};

\node (00) at ( 0,0) [vertex] {};
\node (10) at ( -1,1) [vertex] {};
\node (01) at ( 1,1) [vertex] {};
\node (11) at ( 0,2) [vertex] {};
\node (21) at ( -1,3) [vertex] {};
\node (02) at (2,2) [vertex] {};
\node (12) at (1,3) [vertex] {};
\node (13) at (2,4) [vertex] {};
\node (22) at (0,4) [vertex] {};
\node (23) at (1,5) [vertex] {};
\draw [-] (00) to (10);
\draw [-] (00) to (01);
\draw [-] (10) to (11);
\draw [-] (01) to (11);
\draw [-] (11) to (21);
\draw [-] (01) to (02);
\draw [-] (11) to (12);
\draw [-] (02) to (12);
\draw [-] (21) to (22);
\draw [-] (12) to (22);
\draw [-] (12) to (13);
\draw [-] (22) to (23);
\draw [-] (13) to (23);

\draw [-] (a) to (b);
\draw [-] (a) to (c);
\draw [-] (c) to (d);
\draw [-] (d) to (e);
\draw [-] (b) to (00);
\draw [-] (c) to (00);
\draw [-] (d) to (01);
\draw [-] (e) to (02);

\end{tikzpicture}
\end{center}
\end{example}

We will show that any poset of width two is uniquely determined by its partition generating function. We will first need several useful lemmas about the structure of $P$.

\begin{lemma}\label{P prime}  Let $P'$ be the induced subposet of $P$ consisting of all elements that are not minimal.  The generating function for $P'$ is determined by $\KP$.
\end{lemma}

\begin{proof}  This follows immediately from Lemma \ref{Remove Jump} when $i=1$.  
\end{proof}

In the case when $P$ has width $2$ and is irreducible (and hence has two minimal elements), we can explicitly find the partition generating function for $P'$ as
\[K_{P'}(\mathbf{x}) = (\max_2\nolimits \otimes id) \circ \Delta_{2, n-2} \KP.\]
In terms of $J(P)$, the subposet of elements greater than or equal to $(1, 1)$ is isomorphic to $J(P')$.

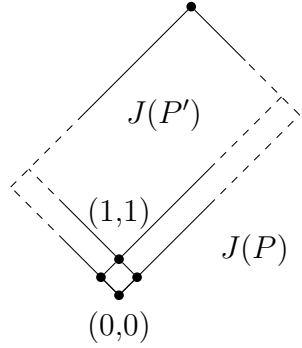
\begin{center}
	\begin{tikzpicture}
	[auto,
	vertex/.style={circle,draw=black!100,fill=black!100,thick,inner sep=0pt,minimum size=1mm},scale=1.2]
	\node(R) at (1.5, .5) {$J(P)$};
	\node(R) at (.5, 2) {$J(P')$};
	\node (min) at ( 0,0) [vertex, label=below:$(0\text{,}0)$] {};
	\node (2) at ( -.2,.2) [vertex] {};
	\node (3) at ( .2,.2) [vertex] {};
	\node (4) at ( 0,.4) [vertex, label={[label distance=.2cm]above:$(1\text{,}1)$}] {};
	%\node(max) at (.8, 3.2) [vertex, label=above:$(l\text{,}m)$] {};
	\node(max) at (.8,3.2) [vertex] {};
	\draw [-] (0,0) to (-.6, .6);
	\draw [dashed] (-.6, .6) to (-1.2, 1.2);
	\draw [-] (0,0) to (1, 1);
	\draw [dashed] (1, 1) to (2, 2);
	\draw [dashed] (-1.2, 1.2) to (-.4, 2);
	\draw [-] (-.4, 2) to (max);
	\draw [-] (1.4, 2.6) to (max);
	\draw [dashed] (2, 2) to (1.4, 2.6);
	\draw [-] (min) to (2);
	\draw [-] (min) to (3);
	\draw [-] (2) to (4);
	\draw [dashed] (-.6, 1) to (-1, 1.4);
	\draw [-] (4) to (-.6, 1);
	\draw [-] (3) to (4);
	\draw [dashed] (1, 1.4) to (1.8, 2.2);
	\draw [-] (4) to (1, 1.4);
	\end{tikzpicture}
\end{center}

It is not true that $P'$ must be irreducible if $P$ is irreducible, but there are some restrictions for what the ordinal sum decomposition of $P'$, or indeed of any filter of $P$, can be.  

%%%%NEW BEGIN
\begin{lemma}\label{Filter decomp}
If $P = C_1 \uplus C_2$ is irreducible, then for all filters $F \subseteq P$, $F$ can be expressed as $F = C \oplus R$, where $C$ is a (possibly empty) chain satisfying $C \subseteq C_1$ or $C \subseteq C_2$, and $R$ is irreducible.
\end{lemma}

\begin{proof}
Let $F$ be a filter of $P$.  We can express $F$ as $F = C \oplus R$, where $R$ is irreducible.  It remains to be shown that $C$ is a chain contained in either $C_1$ or $C_2$.

Suppose that $C$ is not a chain.  This means that the width of $C$ is $2$.  Every element in $P \backslash F$ must be less than an element of $C$ or else the width of $P$ would be at least $3$.  Therefore every element of $R$ is greater than every element of $P\backslash R$ implying that $P$ is reducible.  Therefore $C$ is a chain.

We conclude the proof by showing that either $C \subseteq C_1$ or $C \subseteq C_2$.  Suppose the minimal element of $C$ is an element of $C_1$.  If $C$ contained an element $c_2 \in C_2$, then $c_2$ would be related to all of the elements in $P$ (for any element of $P \backslash F$ either also lies in $C_2$, or it lies in $C_1$ and is less than the minimal element of $C$, which is also less than $c_2$). This cannot happen if $P$ is irreducible.
\end{proof}
%%%%NEW END

We say that an ideal $I$ of $P$ is a \emph{chain ideal} if $I$ is a chain. If $P$ is irreducible, let $a$ and $b$ (assume $a \leq b$) be the sizes of the two maximal chain ideals of $P$. One of these chain ideals will be contained in $C_1$ and the other in $C_2$.  If, say, the largest chain ideal in $C_1$ has $a$ elements, then the $(a+1)$st element of $C_1$ is the smallest element of $C_1$ greater than the minimum element of $C_2$.  The value of $b$ can be described similarly.

In terms of $J(P)$ (which we embed in $\mathbb N^2$ as described above), $a$ and $b$ can be determined by finding the coordinates of the largest points on the $x$- and $y$-axes. For instance, in Example~\ref{jp}, $a=1$ and $b=3$.

\begin{lemma}\label{chain ideal} Let $P=C_1 \uplus C_2$ be irreducible. Then the values of $a$ and $b$ are determined by $\KP$, and if $a \neq b$, then there exists exactly one $(a+1)$-element chain ideal. If $I_{a+1}$ is this chain ideal and $P'' = P \setminus I_{a+1}$, then $K_{P''}(\mathbf{x})$ is determined by $\KP$.
\end{lemma}

\begin{proof}
When the width of $P$ is at most $2$, $P$ has at most $2$ chain ideals of any given size.  Also, since $P$ is irreducible, the only way for an ideal to have exactly one minimal element is if it is a chain ideal (for if an ideal with width two had one minimal element, then this element would be less than every other element of $P$, which cannot happen if $P$ is irreducible).
Since $a\leq b$, the value of $a$ is the largest number such that $P$ has two $a$-element chain ideals, and $b$ is the smallest number such that $P$ has no $(b+1)$-element chain ideals.

The number of $k$-element ideals in $P$ is counted by $\rank_k(\KP) = \sum_{i,j} \anti_{k,i,j}(\KP)$.  The $k$-element chain ideals are exactly the $k$-element ideals in $P$ that do not contain both minimal elements.  We can count the number of $k$-elements ideals of $P$ that contain both minimal elements by counting the number of $(k-2)$-element ideals in $P'$.  This is counted by $\rank_{k-2}(K_{P'}(\mathbf{x}))$.  Therefore the number of $k$-element chain ideals in $P$ is counted by $\rank_k(\KP) - \rank_{k-2}(K_{P'}(\mathbf{x}))$.  Thus the value of $a$ is the largest number such that
\[\rank_a(\KP) - \rank_{a-2}(K_{P'}(\mathbf{x})) = 2,\]
while the value of $b$ is smallest number such that 
\[\rank_{b+1}(\KP) - \rank_{b-1}(K_{P'}(\mathbf{x})) = 0.\]

We will now show that if $a \neq b$, then $K_{P''}(\mathbf{x})$ is determined by $\KP$. Since there is a unique $(a+1)$-element chain ideal and $P$ is irreducible, there is a unique $(a+1)$-element ideal with exactly one minimal element.  Therefore the result follows as in Corollary \ref{P without I}:
\[K_{P''}(\mathbf{x}) = (\min_1\nolimits \otimes id) \Delta_{a+1, n-a-1} \KP. \qedhere\]
\end{proof}

Note that if the width of $P$ is at most two, then the width of any induced subposet is also at most two.  In particular, the widths of $P'$ and $P''$ are both at most two.

\begin{example}
The following is a width 2 poset $P$ along with $P'$ and $P''$. Here, $a=1$ and $b=3$, which correspond to the maximal chain ideals in the left and right chain, respectively. (See also $J(P)$ in Example~\ref{jp}.)
\begin{center}
\begin{tikzpicture}
 [auto,
 vertex/.style={circle,draw=black!100,fill=black!100,thick,inner sep=0pt,minimum size=1mm}]
\node(P) at (-1, 2) {$P$};
\node (v) at ( 0,-1) [vertex] {};
\node (v0) at ( 1,-1) [vertex] {};
\node (v1) at ( 0,0) [vertex] {};
\node (v2) at ( 1,0) [vertex] {};
\node (v4) at ( 0,1) [vertex] {};
\node (v5) at ( 1,1) [vertex] {};
\node (v7) at ( 1,2) [vertex] {};
\draw [-] (v1) to (v4);
\draw [-] (v2) to (v4);
\draw [-] (v2) to (v5);
\draw [-] (v5) to (v7);
\draw [-] (v1) to (v7);
\draw [-] (v) to (v1);
\draw [-] (v0) to (v2);
\draw [-] (v0) to (v1);
\end{tikzpicture}
\hspace{2cm}
\begin{tikzpicture}
 [auto,
 vertex/.style={circle,draw=black!100,fill=black!100,thick,inner sep=0pt,minimum size=1mm}]
\node(P) at (-1, 2) {$P'$};
\node (v1) at ( 0,0) [vertex] {};
\node (v2) at ( 1,0) [vertex] {};
\node (v4) at ( 0,1) [vertex] {};
\node (v5) at ( 1,1) [vertex] {};
\node (v7) at ( 1,2) [vertex] {};
\node at (0,-1){};
\draw [-] (v1) to (v4);
\draw [-] (v2) to (v4);
\draw [-] (v2) to (v5);
\draw [-] (v5) to (v7);
\draw [-] (v1) to (v7);
\end{tikzpicture}
\hspace{2cm}
\begin{tikzpicture}
 [auto,
 vertex/.style={circle,draw=black!100,fill=black!100,thick,inner sep=0pt,minimum size=1mm}]
\node(P) at (-1, 2) {$P''$};
\node (v) at ( 0,-1) [vertex] {};
%\node (v0) at ( 1,-1) [vertex] {};
\node (v1) at ( 0,0) [vertex] {};
%\node (v2) at ( 1,0) [vertex] {};
\node (v4) at ( 0,1) [vertex] {};
\node (v5) at ( 1,1) [vertex] {};
\node (v7) at ( 1,2) [vertex] {};
\draw [-] (v1) to (v4);
%\draw [-] (v2) to (v4);
%\draw [-] (v2) to (v5);
\draw [-] (v5) to (v7);
\draw [-] (v1) to (v7);
\draw [-] (v) to (v1);
%\draw [-] (v0) to (v2);
%\draw [-] (v0) to (v1);
\end{tikzpicture}

\end{center}
\end{example}

We are now ready to prove the main result of this section.

\begin{thm}\label{width2}
If the width of $P$ is at most $2$, then $P$ is uniquely determined by $\KP$. If $P$ is irreducible, then $P$ has a unique decomposition $P = C_1 \uplus C_2$ (up to reordering).
\end{thm}

\begin{proof}
We prove this by induction on the size of $P$.  The case when $P$ has one element is trivial.  By Lemma \ref{K of reducible}, we can assume that $P$ is irreducible.

Now suppose that $P$ is an irreducible width two poset, and assume that the theorem holds for all smaller width two posets.
% Let $|C_1| = l$ and $|C_2| = m$, and assume without loss of generality that $l \leq m$.
By Lemma \ref{P prime}, we can determine the generating function for $P'$, and by induction, $K_{P'}(\mathbf{x})$ uniquely determines $P'$.  Therefore, if $Q$ is a poset such that $\KP = \KQ$, then $P' \cong Q'$.

%By Lemma~\ref{chain ideal}, we know that $P$ has 

%The chain ideals of $P$ correspond to the elements of the form $(k, 0)$ or $(0, k)$ in $J(P)$.
%By Lemma \ref{chain ideal}, we know that $P$ has two $a$-element chain ideals.%, meaning $J(P)$ must contain both the points $(a, 0)$ and $(0, a)$.  
%If $a = b$, then there is only one way to add elements to $J(P')$ to get $J(P)$, so $P$ is uniquely determined by $\KP$.  If $a \neq b$, then $P$ has exactly one $b$-element chain ideal.  Therefore $J(P)$ either contains $(b, 0)$ or $(0, b)$ but not both.  We will continue the proof by showing that the poset that contains $(b, 0)$ has a different generating function than the poset that contains $(0, b)$.  In order to do this, we have to consider the following cases: when $P'$ is irreducible, and when $P'$ is reducible.\\

\medskip
\noindent
\underline{Case 1:} $P'$ is irreducible.

By induction, there is a unique decomposition $P' = C_1' \uplus C_2'$ into two chains (up to reordering). Let $|C_1'| = l-1$ and $|C_2'|=m-1$. If $P = C_1 \uplus C_2$, then $C_1'$ and $C_2'$ must be obtained from $C_1$ and $C_2$ by removing their minimal elements. Since $\KP$ determines $a$ and $b$ by Lemma~\ref{chain ideal}, there are at most two possibilities for how these minimal elements can compare to the elements in the other chain, depending on whether the maximal chain ideal in $C_1$ has $a$ elements or $b$ elements. Let $P$ and $Q$ be the two posets obtained in this way, and suppose $\KP = \KQ$. In terms of $J(P), J(Q) \subset \N^2$, the principal filter generated by $(1,1)$ in either is $J(P') \cong J(Q')$, and 
\begin{align*}
J(P) &= J(P') \cup [(0, 0), (a, 0)] \cup [(0, 0), (0, b)],\\ 
J(Q) &= J(P') \cup [(0, 0), (b, 0)] \cup [(0, 0), (0, a)].
\end{align*}

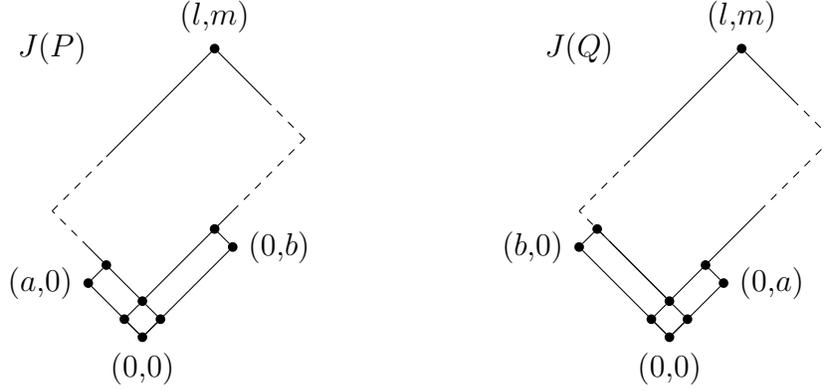
\begin{center}
\begin{tikzpicture}
 [auto,
 vertex/.style={circle,draw=black!100,fill=black!100,thick,inner sep=0pt,minimum size=1mm},scale=1.2]
\node (min) at ( 0,0) [vertex, label=below:$(0\text{,}0)$] {};
\node (2) at ( -.2,.2) [vertex] {};
\node (3) at ( .2,.2) [vertex] {};
\node (4) at ( 0,.4) [vertex] {};
\node (a) at (-.6, .6) [vertex, label=left:$(a\text{,}0)$] {};
\node (b) at (1, 1) [vertex, label=right:$(0\text{,}b)$] {};
\node (5) at (-.4, .8) [vertex] {};
\node (6) at (.8, 1.2) [vertex] {};
\node(max) at (.8, 3.2) [vertex, label=above:$(l\text{,}m)$] {};
\node(P) at (-1, 3.2) {$J(P)$};
\draw [-] (min) to (a);
\draw [-] (min) to (b);
\draw [dashed] (-1, 1.4) to (-.4, 2);
\draw [-] (-.4, 2) to (max);
\draw [-] (1.4, 2.6) to (max);
\draw [dashed] (1.8, 2.2) to (1.4, 2.6);
\draw [-] (min) to (2);
\draw [-] (min) to (3);
\draw [-] (2) to (4);
\draw [dashed] (-.6, 1) to (-1, 1.4);
\draw [-] (4) to (-.6, 1);
\draw [-] (3) to (4);
\draw [dashed] (1, 1.4) to (1.8, 2.2);
\draw [-] (4) to (1, 1.4);
\draw [-] (a) to (5);
\draw [-] (b) to (6);
\end{tikzpicture}
\hspace{2 cm}
\begin{tikzpicture}
 [auto,
 vertex/.style={circle,draw=black!100,fill=black!100,thick,inner sep=0pt,minimum size=1mm},scale=1.2]
\node (min) at ( 0,0) [vertex, label=below:$(0\text{,}0)$] {};
\node (2) at ( -.2,.2) [vertex] {};
\node (3) at ( .2,.2) [vertex] {};
\node (4) at ( 0,.4) [vertex] {};
\node (a) at (.6, .6) [vertex, label=right:$(0\text{,}a)$] {};
\node (b) at (-1, 1) [vertex, label=left:$(b\text{,}0)$] {};
\node (5) at (.4, .8) [vertex] {};
\node (6) at (-.8, 1.2) [vertex] {};
\node(max) at (.8, 3.2) [vertex, label=above:$(l\text{,}m)$] {};
\node(Q) at (-1, 3.2) {$J(Q)$};
\draw [-] (min) to (b);
\draw [-] (min) to (a);
\draw [dashed] (-1, 1.4) to (-.4, 2);
\draw [-] (-.4, 2) to (max);
\draw [-] (1.4, 2.6) to (max);
\draw [dashed] (1.8, 2.2) to (1.4, 2.6);
\draw [-] (min) to (2);
\draw [-] (min) to (3);
\draw [-] (2) to (4);
\draw [dashed] (-.6, 1) to (-1, 1.4);
\draw [-] (4) to (-.6, 1);
\draw [-] (3) to (4);
\draw [dashed] (1, 1.4) to (1.8, 2.2);
\draw [-] (4) to (1, 1.4);
\draw[-] (4) to (6);
\draw [-] (a) to (5);
\draw [-] (b) to (6);
\end{tikzpicture}
\end{center}

If $a=b$, then clearly $P \cong Q$. Otherwise, Lemma \ref{chain ideal} states that we can determine the generating functions for $P''$ and $Q''$ from $\KP = \KQ$, so by induction $P'' \cong Q''$.  In terms of $J(P)$ and $J(Q)$, we have that 
\[J(P) \supset [(0, a+1), (l, m)] \cong [(a+1, 0), (l, m)] \subset J(Q).\]
By Lemma~\ref{Filter decomp}, these subposets have the form $P'' \cong Q'' = C \oplus R$ where $C$ is a (possibly empty) chain contained in one of the two chains of $P$ or $Q$, and $R$ is an irreducible width two poset.  Let $c = |C|$.  (Note $c=0$ unless $b=a+1$.) Since $R$ is a subposet of both $P$ and $Q$,
\begin{align*}
J(R) &\cong [(c, a+1), (l, m)] \subset J(P),\\
J(R) &\cong [(a+1, c), (l, m)] \subset J(Q).
\end{align*}
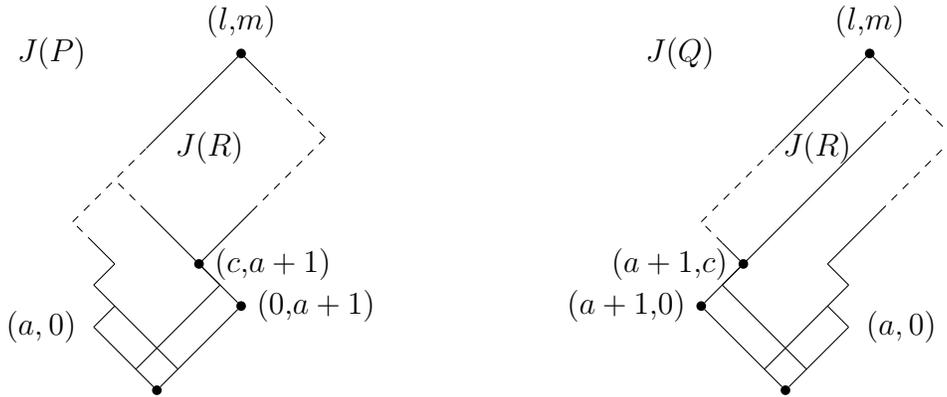
\begin{center}
\begin{tikzpicture}
 [auto,
 vertex/.style={circle,draw=black!100,fill=black!100,thick,inner sep=0pt,minimum size=1mm},scale=1.4]
\node (min) at ( 0,0) [vertex] {};
\node (2) at ( -.2,.2) {};
\node (3) at ( .2,.2) {};
\node (4) at ( 0,.4) {};
\node (a) at (-1.1, .6) {$(a, 0)$};
\node (a1) at (.8, .8) [vertex, label=right:$(0\text{,}a+1)$] {};
\node (cP) at (.4, 1.2) [vertex, label=right:$(c\text{,}a+1)$] {};
\node (b) at (1, 1) {};
\node (5) at (-.4, .8) {};
\node (6) at (.8, 1.2) {};
\node(max) at (.8, 3.2) [vertex, label=above:$(l\text{,}m)$] {};
\node(R) at (.5, 2.3) {$J(R)$};
\node(P) at (-1, 3.2) {$J(P)$};

\draw [-] (-.2, .2) to (0, .4);
\draw [-] (.2, .2) to (0, .4);

\draw [-] (min) to (-.6, .6);
\draw [-] (-.6, .6) to (-.4, .8);

\draw [-] (min) to (a1);
\draw [-] (a1) to (cP);

\draw [-] (0, .4) to (-.6, 1);
\draw [-] (-.6, 1) to (-.4, 1.2); 
\draw [-] (-.4, 1.2) to (-.6, 1.4);
\draw [dashed] (-.6, 1.4) to (-.8, 1.6);
\draw [dashed] (-.8, 1.6) to (-.1, 2.3);
\draw [-] (-.1, 2.3) to (.8, 3.2);

\draw [-] (0, .4) to (.6, 1);
\draw [-] (.6, 1) to (.4, 1.2);
\draw [-] (.4, 1.2) to (.9, 1.7);
\draw [dashed] (.9, 1.7) to (1.6, 2.4);
\draw [dashed] (1.6, 2.4) to (1.1, 2.9);
\draw [-] (1.1, 2.9) to (.8, 3.2);

\draw [-] (.4, 1.2) to (-.1, 1.7);
\draw [dashed] (-.1, 1.7) to (-.4, 2);

\end{tikzpicture}
\hspace{2 cm}
\begin{tikzpicture}
 [auto,
 vertex/.style={circle,draw=black!100,fill=black!100,thick,inner sep=0pt,minimum size=1mm},scale=1.4]
\node (min) at ( 0,0) [vertex] {};
\node (2) at ( -.2,.2) {};
\node (3) at ( .2,.2) {};
\node (4) at ( 0,.4) {};
\node (a) at (1.1, .6) {$(a, 0)$};
\node (a1) at (-.8, .8) [vertex, label=left:$(a+1\text{,}0)$] {};
\node (cQ) at (-.4, 1.2) [vertex, label=left:$(a+1\text{,}c)$] {};
\node (b) at (1, 1) {};
\node (5) at (-.4, .8) {};
\node (6) at (.8, 1.2) {};
\node(max) at (.8, 3.2) [vertex, label=above:$(l\text{,}m)$] {};
\node(R) at (.3, 2.3) {$J(R)$};
\node(Q) at (-1, 3.2) {$J(Q)$};

\draw [-] (-.2, .2) to (0, .4);
\draw [-] (.2, .2) to (0, .4);

\draw [-] (min) to (.6, .6);
\draw [-] (.6, .6) to (.4, .8);

\draw [-] (min) to (a1);
\draw [-] (a1) to (cQ);

\draw [-] (0, .4) to (-.6, 1);
\draw [-] (-.6, 1) to (-.4, 1.2); 
\draw [-] (-.4, 1.2) to (-.6, 1.4);
\draw [dashed] (-.6, 1.4) to (-.8, 1.6);
\draw [dashed] (-.8, 1.6) to (-.1, 2.3);
\draw [-] (-.1, 2.3) to (.8, 3.2);

\draw [-] (0, .4) to (.6, 1);
\draw [-] (.6, 1) to (.4, 1.2);
\draw [-] (.4, 1.2) to (.9, 1.7);
\draw [dashed] (.9, 1.7) to (1.6, 2.4);
\draw [dashed] (1.6, 2.4) to (1.1, 2.9);
\draw [-] (1.1, 2.9) to (.8, 3.2);

\draw (-.4, 1.2) to (.9, 2.5);
\draw [dashed] (.9, 2.5) to (1.2, 2.8);

\end{tikzpicture}
\end{center}

Both of these embeddings of $J(R)$ correspond to a partition of $R$ into two chains.  By induction, since $R$ is irreducible, the partition of $R$ into two chains is unique up to reordering, which corresponds to a reflection of $J(R)$.

If $J(R)$ is embedded in the same way in both $J(P)$ and $J(Q)$, then $c=a+1$. But then in $J(P') \cong J(Q')$, $(a+1,a+1)$ is the only element in its rank, contradicting irreducibility.

Otherwise the embeddings of $J(R)$ in $J(P)$ and $J(Q)$ are reflections of one another, that is, the isomorphism between $[(c, a+1), (l, m)] \subset J(P)$ and $[(a+1, c), (l, m)] \subset J(Q)$ must be $(x, y) \leftrightarrow (y, x)$. But this implies that $J(P')$ is symmetric, so we can extend this isomorphism to get $J(P) \cong J(Q)$. Hence $P \cong Q$, and the isomorphism corresponds to a reordering of the two chains.

\medskip
\noindent
\underline{Case 2:} $P'$ is reducible.

By Lemma \ref{Filter decomp}, $P' = C \oplus R$ where $C$ is a nonempty chain and $R$ is irreducible.  Since $R$ is irreducible it can be partitioned uniquely into two chains $A$ and $B$.
% such that $A \subset C_1$, $B \subset C_2$, and either $C \subset C_1$ or $C \subset C_2$.
Suppose $|A| = j-1$, $|B| = k-1$, and $|C| = c \geq 1$. If $P = C_1 \uplus C_2$, with $A \subset C_1$ and $B \subset C_2$, then by Lemma~\ref{Filter decomp}, either $C \subset C_1$ or $C \subset C_2$. Again, by Lemma~\ref{chain ideal}, $a$ and $b$ are determined. In fact, we must have $a=1$ (the maximal chain ideal in the chain not containing $C$ can only have size $1$) and $b>1$.

There are again two possibilities, so let $P$ be the poset where $C \subset C_1$, and let $Q$ be the poset where $C \subset C_2$.  The subposet $J(R)$ must be isomorphic to both of the intervals $[(c+1, 1), (c+j, k)] \subset J(P)$ and $[(1, c+1), (j, c+k)] \subset J(Q)$.

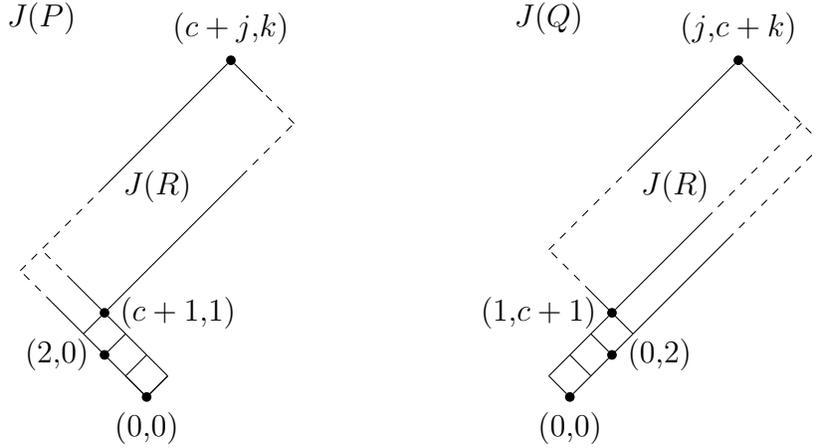
\begin{center}
\begin{tikzpicture}
 [auto,
 vertex/.style={circle,draw=black!100,fill=black!100,thick,inner sep=0pt,minimum size=1mm},scale=1.4]
\node(JP) at (-1, 3.6) {$J(P)$};
\node(R) at (.1, 2) {$J(R)$};
\node (min) at ( 0,0) [vertex, label=below:$(0\text{,}0)$] {};
\node (5) at (-.4, .4) [vertex, label=left:$(2\text{,}0)$] {};
\node(max) at (.8, 3.2) [vertex, label=above:$(c+j\text{,}k)$] {};
\node (c) at (-.4, .8) [vertex, label=right:$(c+1\text{,}1)$] {};
\draw [-] (0,0) to (-.6, .6);
\draw [-] (-.6, .6) to (c);
\draw [-] (c) to (-.7, 1.1);
\draw [dashed] (-.7, 1.1) to (-1, 1.4);
\draw [-] (0,0) to (.2, .2);
\draw [-] (0, .4) to (c);
\draw [-] (-.6, .6) to (-.9, .9);
\draw [dashed] (-.9, .9) to (-1.2, 1.2);
\draw [-] (5) to (-.2, .6);
\draw [dashed] (-1.2, 1.2) to (-.4, 2);
\draw [-] (-.4, 2) to (max);
\draw [-] (1.1, 2.9) to (max);
\draw [-] (c) to (.9, 2.1);
\draw [dashed] (.9,2.1) to (1.4, 2.6);
\draw [dashed] (1.4, 2.6) to (1.1, 2.9);
\draw [-] (min) to (-.2, .2);
\draw [-] (min) to (.2, .2);
\draw [-] (-.2, .2) to (0, .4);
\draw [-] (.2, .2) to (0, .4);
\end{tikzpicture}
\hspace{2 cm}
\begin{tikzpicture}
 [auto,
 vertex/.style={circle,draw=black!100,fill=black!100,thick,inner sep=0pt,minimum size=1mm},scale=1.4]
\node(JQ) at (-.2, 3.6) {$J(Q)$};
\node(R) at (1, 2) {$J(R)$};
\node (min) at (0,0) [vertex, label=below:$(0\text{,}0)$] {};
\node (5) at (.4, .4) [vertex, label=right:$(0\text{,}2)$] {};
\node (c) at (.4, .8) [vertex, label=left:$(1\text{,}c+1)$] {};
\node (max) at (1.6, 3.2) [vertex, label=above:$(j\text{,}c+k)$] {};
\draw [-] (min) to (-.2, .2);
\draw [-] (-.2, .2) to (c);
\draw [-] (.2, .2) to (0, .4);
\draw [-] (min) to (.6, .6);
\draw [-] (.6, .6) to (1.5, 1.5);
\draw [-] (5) to (.2, .6);
\draw [dashed] (1.5, 1.5) to (2.4, 2.4);
\draw [-] (c) to (1.3, 1.7);
\draw [dashed] (1.3, 1.7) to (2.2, 2.6);
\draw [dashed] (2.4, 2.4) to (2, 2.8);
\draw [-] (2, 2.8) to (max);
\draw [dashed] (.2, 1) to (-.2, 1.4);
\draw [-] (.6, .6) to (.2, 1);
\draw [-] (.6, 2.2) to (max);
\draw [dashed] (-.2, 1.4) to (.6, 2.2);
\end{tikzpicture}
\end{center}

By Lemma \ref{chain ideal}, we have that $J(P'') \cong J(Q'')$, that is to say,
\[J(P) \supset [(2, 0), (c+j, k)] \cong [(0, 2), (j, c+k)] \subset J(Q).\] Also $P''$ and $Q''$ must be irreducible. (If they were reducible, then there would be only one rank $3$ element of $J(P)$ and $J(Q)$, namely, $(2, 1)$ and $(1, 2)$, respectively.) This can only happen if the isomorphism is a translation or if it is a reflection.

If the isomorphism is a translation, then $c=2$, and the translation is $(x, y) \leftrightarrow (x-2, y+2)$.  However, $J(P'')$ and $J(Q'')$ are not translations of each other.  We know this because $(3, 1) \in J(P)$ and $P$ is irreducible, so we must also have the rank $4$ element $(4,0) \in J(P'')$, but $(2, 2) \notin J(Q'')$.  Therefore this possibility cannot happen.

If the isomorphism is a reflection, then $j=k$, and the isomorphism is $(x, y) \leftrightarrow (y, x)$.  Since 
\begin{align*}
J(P) &= J(P'') \cup \{(0, 0), (1, 0), (0, 1), (1, 1)\},\\
J(Q) &= J(Q'') \cup \{(0, 0), (1, 0), (0, 1), (1, 1)\},
\end{align*}
the isomorphism between $J(P'')$ and $J(Q'')$ can be extended to $J(P)$ and $J(Q)$. This isomorphism corresponds to a reordering of the chains $C_1$ and $C_2$.
%Therefore $P$ is uniquely determined by $\KP$ and $P$ can be partition uniquely into two chains.  Since the result holds for all irreducible posets whose width is at most $2$, Lemma \ref{K of reducible} implies the result.
\end{proof}

Theorem \ref{width2} tells us that any poset $P$ whose shape $\lambda$ has at most two parts has a unique $P$-partition generating function.

%%%%%%%%%%%%%%%%%%%%%%%%%%%%%%%%%%%%%%%%%%%%%%%%%%%%%
\subsection{Hook shaped posets}
A partition $\lambda$ is said to be \emph{hook shaped} if $\lambda_2 \leq1$.  Hook shaped partitions are therefore of the form $\lambda = (\lambda_1, 1, 1, \dots, 1)$. In this section, we will show that a poset whose shape is a hook is determined not just by $\KP$ but by $\supp_L(\KP)$.
\begin{thm}\label{Hook Shaped}
If $\sh(P)$ is hook shaped, then $P$ is determined by $\supp_L(\KP)$, that is, if $\supp_L(\KP) = \supp_L(\KQ)$, then $P \cong Q$.
\end{thm}

\begin{proof}
If $\sh(P)$ is hook shaped, then $P$ can be expressed as the union of a chain $C$ and an antichain $A$ where $|C \cap A| = 1$.  The jump pair of each element in $C$ is determined by its position in the chain.  Each element of $A$ can cover at most one element of $C$ and is covered by at most one element of $C$.  For each $a \in A$, $\jumppair(a)$ is determined by the element that $a$ covers and the element that covers $a$.  This implies that hook shaped posets are determined by the jump pairs of their elements.  Since the multiset of $\jumppair(a)$ for all $a$ is determined by $\supp_L(\KP)$ by Lemma \ref{jumppair}, it follows that if $\sh(P)$ is hook shaped and $\supp_L(\KP) = \supp_L(\KQ)$, then $P \cong Q$.
\end{proof}

\begin{cor} If $\sh(P)$ is hook shaped and $K_P(\mathbf{x}) = K_Q(\mathbf{x})$, then $P \cong Q$.
\end{cor}

\begin{example}\label{hook example}
Consider the following two hook shaped posets.

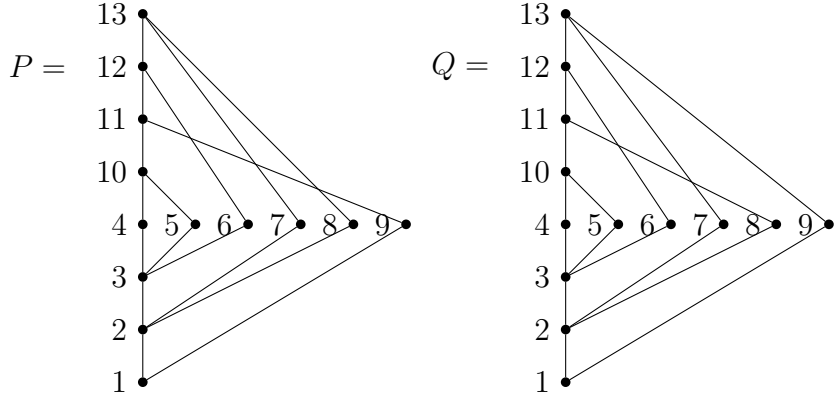
\begin{center}

\begin{tikzpicture}
 [auto,
 vertex/.style={circle,draw=black!100,fill=black!100,thick,inner sep=0pt,minimum size=1mm}, scale =.7]
\node(P) at (-2, 6) {$P =$};
\node (1) at (0,0) [vertex,label=left:$1$] {};
\node (2) at (0,1) [vertex,label=left:$2$] {};
\node (3) at (0,2) [vertex,label=left:$3$] {};
\node (4) at (0,3) [vertex,label=left:$4$] {};
\node (5) at (1,3) [vertex,label=left:$5$] {};
\node (6) at (2,3) [vertex,label=left:$6$] {};
\node (7) at (3,3) [vertex,label=left:$7$] {};
\node (8) at (4,3) [vertex,label=left:$8$] {};
\node (9) at (5,3) [vertex,label=left:$9$] {};
\node (10) at (0,4) [vertex,label=left:$10$] {};
\node (11) at (0,5) [vertex,label=left:$11$] {};
\node (12) at (0,6) [vertex,label=left:$12$] {};
\node (13) at (0,7) [vertex,label=left:$13$] {};
\draw [-] (1) to (2);
\draw [-] (1) to (9);
\draw [-] (2) to (7);
\draw [-] (2) to (8);
\draw [-] (2) to (3);
\draw [-] (3) to (4);
\draw [-] (3) to (5);
\draw [-] (3) to (6);
\draw [-] (4) to (10);
\draw [-] (5) to (10);
\draw [-] (6) to (12);
\draw [-] (7) to (13);
\draw [-] (8) to (13);
\draw [-] (9) to (11);
\draw [-] (10) to (11);
\draw [-] (11) to (12);
\draw [-] (12) to (13);
\end{tikzpicture}
\begin{tikzpicture}
 [auto,
 vertex/.style={circle,draw=black!100,fill=black!100,thick,inner sep=0pt,minimum size=1mm}, scale = .7]
\node(Q) at (-2, 6) {$Q =$};
\node (1) at (0,0) [vertex,label=left:$1$] {};
\node (2) at (0,1) [vertex,label=left:$2$] {};
\node (3) at (0,2) [vertex,label=left:$3$] {};
\node (4) at (0,3) [vertex,label=left:$4$] {};
\node (5) at (1,3) [vertex,label=left:$5$] {};
\node (6) at (2,3) [vertex,label=left:$6$] {};
\node (7) at (3,3) [vertex,label=left:$7$] {};
\node (8) at (4,3) [vertex,label=left:$8$] {};
\node (9) at (5,3) [vertex,label=left:$9$] {};
\node (10) at (0,4) [vertex,label=left:$10$] {};
\node (11) at (0,5) [vertex,label=left:$11$] {};
\node (12) at (0,6) [vertex,label=left:$12$] {};
\node (13) at (0,7) [vertex,label=left:$13$] {};
\draw [-] (1) to (2);
\draw [-] (1) to (9);
\draw [-] (2) to (7);
\draw [-] (2) to (8);
\draw [-] (2) to (3);
\draw [-] (3) to (4);
\draw [-] (3) to (5);
\draw [-] (3) to (6);
\draw [-] (4) to (10);
\draw [-] (5) to (10);
\draw [-] (6) to (12);
\draw [-] (7) to (13);
\draw [-] (8) to (11);
\draw [-] (9) to (13);
\draw [-] (10) to (11);
\draw [-] (11) to (12);
\draw [-] (12) to (13);
\end{tikzpicture}
\end{center}
The partition generating functions for these posets do not have the same $L$-support because the element $9 \in P$ has jump 1 and up-jump 3, but no element in $Q$ has this jump pair.
\end{example}

\subsection{Nearly hook shaped posets}

In this section, we will show that if the shape of a poset $P$ is \emph{nearly hook shaped}, that is, if $\sh(P) = (\lambda_1, 2, 1, \dots, 1)$, then $P$ is uniquely determined by $\KP$.

\begin{lemma}
	Any finite poset $P$ has a unique antichain $A$ of maximum size such that any other antichain of maximum size is contained in the order ideal $I(A)$ generated by $A$.
\end{lemma}
\begin{proof}
	By Dilworth's theorem, the minimum number of chains into which $P$ can be partitioned is the maximum size of an antichain of $P$. Hence in any such partition, each chain must contain one element from every antichain of maximum size. Then let $A$ consist of the largest element in each chain that is contained in some antichain of maximum size. This is our desired antichain: if $x,y \in A$, then we cannot have $x \succeq y$ since $x$ is incomparable to some $y' \preceq y$.
\end{proof}

Let $A$ be the unique maximum antichain of $P$ as described above. By Theorem~\ref{Anti}, both $|A| = \lambda_1'$ and $|I(A)| = m$ are determined by $\KP$.  Let $P^-$ be the subposet of $P$ consisting of elements less than $A$ in $P$, so that $P^- = I(A) \backslash A$.  The partition generating function for $P^-$ is 
\[K_{P^-}(\mathbf{x}) = (id \otimes \min_{\lambda_1'}\nolimits)(\Delta_{m-\lambda_1', n-(m-\lambda_1')}(\KP)).\]
When $\lambda=\sh(P)$ is nearly hook shaped, $P^-$ must either be a chain, or it can be partitioned into a chain and a single element $x$.  Since the width of $P^-$ is less than or equal to $2$, $P^-$ is determined by $K_{P^-}(\mathbf{x})$ and hence by $\KP$.

Similarly, let $P^+$ be the subposet of $P$ consisting of elements greater than an element of $A$ in $P$.  As with $P^-$, the width of $P^+$ is less than or equal to $2$, so $P^+$ is determined by $K_{P^+}(\mathbf{x})$, which is also determined by $\KP$ by Corollary~\ref{P without I}.

Since $\lambda$ is nearly hook shaped, it cannot be the case that both $P^-$ and $P^+$ have width two.  We will say that:
\begin{enumerate}[(i)]
 \item $P$ is Type 1 if width($P^-$) = 2 and width($P^+$) $\leq$ 1,
 \item $P$ is Type 2 if width($P^-$) $\leq$ 1 and width($P^+$) = 2, 
 \item $P$ is Type 3 if width($P^-$) $\leq$ 1 and  width($P^+$) $\leq$ 1.
\end{enumerate}

\begin{example}
The three posets below each have shape $(5, 2, 1, 1)$.  The poset on left is Type $1$ with maximal antichain $\{5, 6, 7, 8\}$.  The poset in the center is Type $2$ with maximal antichain $\{2, 3, 4, 5\}$.  The poset on the right is Type $3$ with maximal antichain $\{3, 4, 5, 6\}$.
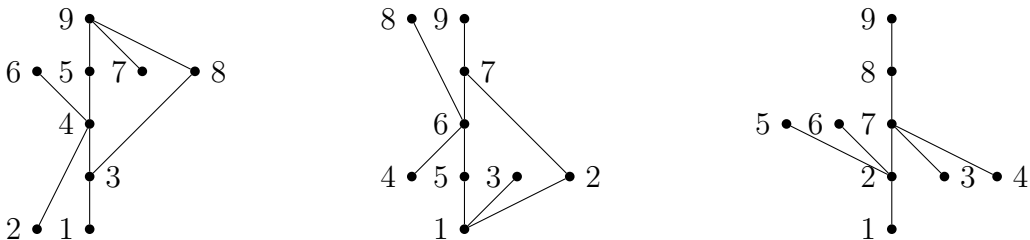
\begin{center}
\begin{tikzpicture}
 [auto,
 vertex/.style={circle,draw=black!100,fill=black!100,thick,inner sep=0pt,minimum size=1mm}, scale =.7]
\node (1) at (0,0) [vertex,label=left:$1$] {};
\node (2) at (-1,0) [vertex,label=left:$2$] {};
\node (3) at (0,1) [vertex,label=right:$3$] {};
\node (4) at (0,2) [vertex,label=left:$4$] {};
\node (5) at (0,3) [vertex,label=left:$5$] {};
\node (6) at (-1,3) [vertex,label=left:$6$] {};
\node (7) at (1,3) [vertex,label=left:$7$] {};
\node (8) at (2,3) [vertex,label=right:$8$] {};
\node (9) at (0,4) [vertex,label=left:$9$] {};
\draw [-] (1) to (3);
\draw [-] (3) to (4);
\draw [-] (4) to (5);
\draw [-] (5) to (9);
\draw [-] (2) to (4);
\draw [-] (4) to (6);
\draw [-] (3) to (8);
\draw [-] (7) to (9);
\draw [-] (8) to (9);
\end{tikzpicture}
%%%%%%%
\hspace{1.5 cm}
\begin{tikzpicture}
 [auto,
 vertex/.style={circle,draw=black!100,fill=black!100,thick,inner sep=0pt,minimum size=1mm}, scale =.7]
\node (1) at (0,4) [vertex,label=left:$9$] {};
\node (2) at (-1,4) [vertex,label=left:$8$] {};
\node (3) at (0,3) [vertex,label=right:$7$] {};
\node (4) at (0,2) [vertex,label=left:$6$] {};
\node (5) at (0,1) [vertex,label=left:$5$] {};
\node (6) at (-1,1) [vertex,label=left:$4$] {};
\node (7) at (1,1) [vertex,label=left:$3$] {};
\node (8) at (2,1) [vertex,label=right:$2$] {};
\node (9) at (0,0) [vertex,label=left:$1$] {};
\draw [-] (1) to (3);
\draw [-] (3) to (4);
\draw [-] (4) to (5);
\draw [-] (5) to (9);
\draw [-] (2) to (4);
\draw [-] (4) to (6);
\draw [-] (3) to (8);
\draw [-] (7) to (9);
\draw [-] (8) to (9);
\end{tikzpicture}
%%%%%%%
\hspace{1.5 cm}
\begin{tikzpicture}
 [auto,
 vertex/.style={circle,draw=black!100,fill=black!100,thick,inner sep=0pt,minimum size=1mm}, scale =.7]
\node (1) at (0,0) [vertex,label=left:$1$] {};
\node (2) at (0,1) [vertex,label=left:$2$] {};
\node (3) at (1,1) [vertex,label=right:$3$] {};
\node (4) at (2,1) [vertex,label=right:$4$] {};
\node (5) at (-2,2) [vertex,label=left:$5$] {};
\node (6) at (-1,2) [vertex,label=left:$6$] {};
\node (7) at (0,2) [vertex,label=left:$7$] {};
\node (8) at (0,3) [vertex,label=left:$8$] {};
\node (9) at (0,4) [vertex,label=left:$9$] {};
\draw [-] (1) to (2);
\draw [-] (2) to (7);
\draw [-] (3) to (7);
\draw [-] (4) to (7);
\draw [-] (2) to (5);
\draw [-] (2) to (6);
\draw [-] (7) to (8);
\draw [-] (8) to (9);
\end{tikzpicture}
\end{center}
\end{example}

Since we can determine the widths of $P^-$ and $P^+$ from $K_{P^-}(\mathbf x)$ and $K_{P^+}(\mathbf x)$, the type of $P$ is determined by $\KP$. Note that the dual of a Type 2 poset is Type 1, so if we can show that Type 1 posets are determined by their $P$-partition generating functions, then Type 2 posets will be as well. (To get the generating function for the dual of a poset $P$, reverse each composition in the expansion of $\KP$ in the $\{M_\alpha\}$-basis.)

\begin{lemma} \label{below antichain}
If $P$ is a Type 1 poset, then $I(A)$ is determined by $\KP$ up to isomorphism.

Moreover, if $x \in P^-$ does not lie in a maximum length chain of $P$, then the number of elements in $A$ that only cover $x$ is determined by $\KP$.
\end{lemma}

\begin{proof}
Suppose that $P$ is Type 1 and that the result holds for all Type 1 posets with fewer elements that $P$.  Since $P$ is Type 1, a maximum length chain in $P$ intersects $P^-$ in an $l$-element chain $C$, with a single element $x \in P^-$ remaining.  Label the elements of $C$ by $1, 2, \dots, l$ from bottom to top. Suppose that $x$ covers $i^*-1$ and $x$ is covered by $j^*$.  If $x$ is minimal, then we let $i^*=1$, while if $x$ is maximal in $P^-$, then we let $j^* = l+1$.

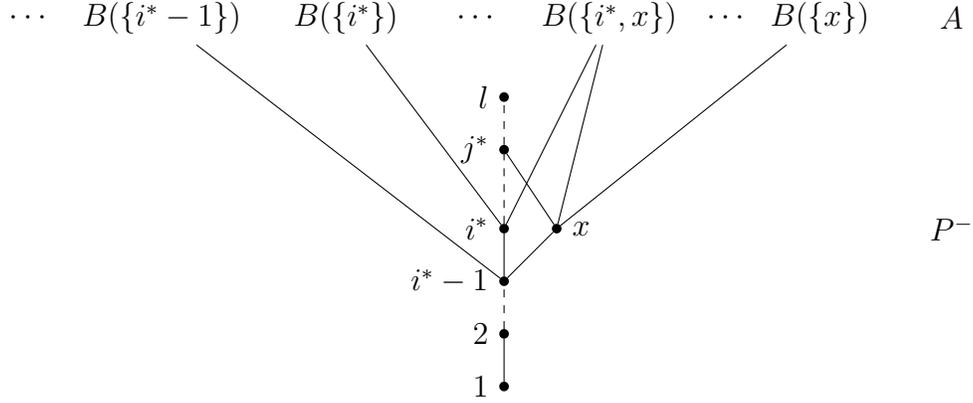
\begin{center}
\begin{tikzpicture}
 [auto, roundnode/.style={circle, draw=black!100, very thick, minimum size=10mm, scale = .5}, 
 vertex/.style={circle,draw=black!100,fill=black!100,thick,inner sep=0pt,minimum size=1mm}, scale = .35]
 
 \node (P) at (17, 6) {$P^-$};
 \node (A) at (17, 14) {$A$};
 \node (1) at (0,0) [vertex,label=left:$1$] {};
 \node (2) at (0,2) [vertex,label=left:$2$] {};
 \node (i star minus 1) at ( 0, 4) [vertex,label=left:$i^*-1$] {};
 \node (i star) at (0, 6) [vertex,label=left:$i^*$] {};
 \node (x) at (2, 6) [vertex,label=right:$x$] {};
 \node (j star) at (0, 9) [vertex,label=left:$j^*$] {};
 \node (l) at (0,11) [vertex,label=left:$l$] {};
 \node at (-18,14) {$\cdots$};
 \node (bi-1) at (-13,14) {$B(\{i^*-1\})$};
 \node (bi) at (-6,14) {$B(\{i^*\})$};
 \node at (-1,14) {$\cdots$};
 \node (bix) at (4,14) {$B(\{i^*,x\})$};
 \node at (8.5,14) {$\cdots$};
 \node (bx) at (12,14) {$B(\{x\})$};

 \draw [-] (1) to (2);
 \draw [dashed] (2) to (i star minus 1);
 \draw [-] (i star minus 1) to (i star);
 \draw [-] (i star minus 1) to (x);
 \draw [dashed] (i star) to (j star);
 \draw [-] (x) to (j star);
 \draw [dashed] (j star) to (l);
 \draw (bi-1)--(i star minus 1);
 \draw (bi)--(i star);
 \draw (i star)--(bix)--(x);
 \draw (bx)--(x); 
\end{tikzpicture}
\end{center}

For each antichain $S \subseteq P^-$, define $B(S)$ to be the set of elements in $A$ that cover the elements of $S$ and no other elements. (By convention, $B(\{0\}) = B(\varnothing)$.) Determining $I(A)$ is equivalent to finding the values $b(S) = |B(S)|$ for all $S$.

Lemma \ref{jumppair} and Theorem \ref{Anti} state that the following statistics on $P$ are determined by $\KP$: (i) the number of elements of $P$ whose jump is $i$ for all $i$; and (ii) the number of elements of $P$ whose principal order ideal has $i+1$ elements.  These statistics can be counted in the following way for $i \leq l$:

\begin{enumerate}[(i)]
\item
\[
\#\{p \in P \mid \jump(p) = i\} = 
\begin{cases}
            b(\{i\}) + 1 & \text{if } i < i^*-1, \\
            b(\{i^*-1\}) + 2 & \text{if } i = i^*-1, \\
            b(\{i^*\}) + b(\{i^*, x\}) + b(\{x\}) + 1 & \text{if } i = i^* , \\
            b(\{i\}) + b(\{i , x\})  + 1 & \text{if } i^* < i < j^* , \\
            b(\{i\}) + 1 & \text{if }  j^* \leq i \leq l .
\end{cases}
\]
\item
\[
\sum_j \anti_{i+1, 1, j}(P)= 
\begin{cases}
            b(\{i\}) + 1 & \text{if } i < i^*-1, \\
            b(\{i^*-1\}) + 2 & \text{if } i = i^*-1, \\
            b(\{i^*\}) + b(\{x\}) + 1 & \text{if } i = i^* , \\
            b(\{i\}) + b(\{i-1 , x\})  + 1 & \text{if } i^* < i < j^* , \\
            b(\{i-1\}) + 1 & \text{if }  j^* \leq i \leq l .
\end{cases}
\]
\end{enumerate}

Since these statistics are all determined by $\KP$, by plugging in all possible values of $i$ and solving, one can determine the values of $b(\{i\})$ for all $i \neq i^*$, $b(\{i , x\})$ for all $i$, and the value of $b(\{i^*\}) + b(\{x\})$. It remains to be shown that $b(\{i^*\})$ and $b(\{x\})$ can be determined by $\KP$.

\medskip
\noindent
\underline{Case 1:} $j^* > i^*+1$.

Let $\hat{P}$ be the poset formed by removing all elements with jump less than $i^* -1$ from $P$.  By Lemma \ref{Remove Jump}, $K_{\hat{P}}(\mathbf{x})$ is determined by $\KP$.
%, and by induction, $\hat{P}$ is uniquely determined by $K_{\hat{P}}(\mathbf{x})$.
The set of minimal elements of $\hat{P}$ is $\{i^*, x\} \cup B(\{i^*-1\})$.  By Theorem \ref{Anti}, we can determine the values of $j$ such that $\anti_{1, 1, j}(\hat{P}) \neq 0$.  Such $j$ (counted with multiplicity $\anti_{1,1,j}(\hat{P})$) are the number of minimal elements remaining when one of the minimal elements of $\hat{P}$ is removed.  These values are $b(\{i^*-1\}) +1$ for removing an element of $B(\{i^*-1\})$, $b(\{i^*-1\}) + b(\{x\}) +1$ for removing $x$, and $b(\{i^*-1\}) + b(\{i^*\}) +2$ for removing $i^*$.  Since we have already determined the value of $b(\{i^*-1\})$, we can determine the set $\{b(\{x\}), b(\{i^*\})+1\}$.  If these values are equal, then we can determine $b(\{x\})$ and  $b(\{i^*\})$.  We will now assume that these values are not equal.

Let $M = \max\{b(\{x\}), b(\{i^*\})+1\}$, and consider the ideal $I \subseteq \hat{P}$ that has $b(\{i^*-1\}) + M + 2$ elements, $b(\{i^*-1\}) + M+1$ of which are maximal.  It is either the case that the maximal elements of $I$ are $B(\{i^*-1\}) \cup B(\{x\}) \cup \{i^*\}$, or $B(\{i^*-1\}) \cup B(\{i^*\}) \cup \{i^*+1, x\}$.  Since there is no other ideal with the same cardinality and number of maximal elements as $I$, Corollary \ref{P without I} says that $K_{\hat{P}\backslash I}(\mathbf{x})$ is determined by $K_{\hat{P}}(\mathbf{x})$.  Observe that $\hat{P}\backslash I$ is hook shaped, so by Theorem \ref{Hook Shaped} it is uniquely determined by $K_{\hat{P}\backslash I}(\mathbf{x})$.

If the maximal elements of $I$ are $B(\{i^*-1\}) \cup B(\{x\}) \cup \{i^*\}$, then the length of the longest chain in $\hat{P}\backslash I$ is $\lambda_1 - i^*$.  Similarly, if the maximal elements of $I$ are $B(\{i^*-1\}) \cup B(\{i^*\}) \cup \{i^*+1, x\}$, then the length of the longest chain in $\hat{P}\backslash I$ is $\lambda_1 - i^* -1$.  Since $\hat{P}\backslash I$ is determined, we can find the length of its longest chain, which allows us to distinguish $b(\{x\})$ and $b(\{i^*\})+1$.

\medskip
\noindent
\underline{Case 2:} $j^* = i^*+1$.

This case follows similarly to Case 1, but the set we can determine is $\{b(\{x\}), b(\{i^*\})\}$.  Since there is an automorphism of $P^-$ that switches $x$ and $i^*$, this is enough to determine $I(A)$ up to isomorphism. However, if $i^*=l$, then $x$ and $i^*$ are both maximal in $P^-$, so $x$ may not lie in a maximum length chain of $P$. In this case, we need to determine $b(\{x\})$, so assume $b(\{x\}) \neq b(\{i^*\})$, and let $M=\max\{b(\{x\}), b(\{i^*\})\}$.

Again, as in Case 1, let $\hat{P}$ be the poset of elements of $P$ of jump at least $i^*-1$. Then there is a unique ideal $I \subseteq \hat P$ with $b(\{i^*-1\}) + M + 2$ elements, $b(\{i^*-1\}) + M+1$ of which are maximal: either the maximal elements are $B(\{i^*-1\}) \cup B(\{x\}) \cup \{i^*\}$ or $B(\{i^*-1\}) \cup B(\{i^*\}) \cup \{x\}$. As in Case $1$, $K_{\hat P \setminus I} (\mathbf x)$ is determined by $\KP$, so we can determine the length of the longest chain in $\hat P \setminus I$. When $x$ does not lie in a maximum length chain of $P$, we must have $b(\{x\}) = M$ if the longest chain in $\hat P \setminus I$ has $\lambda_1' - i^*$ elements, while $b(\{i^*\}) = M$ if the longest chain in $\hat P \setminus I$ has $\lambda_1'-i^*-1$ elements.
\end{proof}

Note that in Lemma~\ref{below antichain}, we chose the element $i^*$ to lie in the longest chain of $P$, so $\upjump(i^*) = \lambda_1 - i^* \geq \upjump(x)$ in $P$. While $x$ must be smaller than some element of $A$ (by maximality of $A$), it may be maximal in $P^-$. In this case, we need to determine the smallest element of the chain $P^+$ that is greater than $x$ (if there is one).

We use the notation $|\V_a|$ to denote the cardinality of the principal filter whose minimum element is $a$.

\begin{lemma}\label{upjump of x}
If $P$ is a Type 1 poset, and $x \in P^-$ is not contained in a maximum length chain of $P$, then the smallest element of $P^+$ that is greater than $x$ is determined by $\KP$.
\end{lemma}

\begin{proof}
We will prove this by induction on the size of $P$.  %The smallest Type 1 posets have $4$ elements and all posets with $4$ elements are uniquely determined by their partition generating function.
Suppose the statement holds for all Type 1 posets with fewer elements than $P$. If $x$ is not maximal in $P^-$, then the statement is trivial, so we will assume that $x$ is maximal in $P^-$. 

Recall from Lemma~\ref{jumppair} and Lemma~\ref{jumpideal} that the multiset of up-jump values of the elements with a fixed jump is determined by $\KP$, as is the multiset of  $|\V_a|$ of the elements $a$ with a fixed jump. Using the notation of Lemma~\ref{below antichain}, this implies that we can determine from $\KP$ the multisets
\begin{align*}
S_1 &= \{\upjump(a)+1 \mid a \in B(\{i^*-1\})\} \cup \{\upjump(i^*)+1, \upjump(x)+1\},\\
S_2 &= \{|V_a| \mid a \in B(\{i^*-1\})\} \cup \{|V_{i^*}|, |V_x|\}.
\end{align*}

We know $\upjump(i^*) +1= \lambda_1 - i^*+1$, and $|V_{i^*}|$ is determined by Lemma~\ref{below antichain}. Moreover, for all elements $a \in B(\{i^*-1\})$, $|\V_a| = \upjump(a)+1$.  Therefore, if we compare $S_1$ and $S_2$, we will be able to determine $|\V_{x}|$ (and therefore the number of elements of $P^+$ that are greater than $x$) in all cases except when $|\V_x| = \upjump(x)+1$.

In this exceptional case, the principal filter of $\{x\}$ is a chain. Then we can determine $\upjump(x)$ by considering the poset $\hat{P}$ formed by removing the maximal elements from $P$. We can determine $K_{\hat{P}}(\mathbf x)$ from $\KP$ (by the dual of Lemma~\ref{P prime}), and the shape of $\hat{P}$ is either hook shaped or it is nearly hook shaped.

If $\sh(\hat{P})$ is hook shaped, then $x$ is covered by exactly one element in $A$ and that element is maximal in $P$.  This implies that $x$ is not related to any element in $P^+$.

If $\sh(\hat{P})$ is nearly hook shaped, then by induction we know the smallest element in $\hat{P}^+$ that is greater than $x$.  This is the same element that $x$ is less than in $P^+$.
\end{proof}

If every element of $P^-$ is contained in a maximum length chain of $P$, then there is an automorphism of $P^-$ switching $i^*$ and $x$. In this case, we will need a way of distinguishing the elements $i^*$ and $x$ in Lemma~\ref{below antichain} if $b(\{i^*\}) \neq b(\{x\})$.

\begin{lemma} \label{Near Hook P minus I}
Suppose there is an automorphism of $P^-$ that switches $i^*$ and $x$, and $b(\{i^*\}) > b(\{x\})$. Then the multiset $\{\upjump(a) \mid a \in B(\{x\}) \cup B(\{i^*,x\})\}$ is determined by $\KP$.
% $P$ has a unique ideal $I$ such that $|I| = \displaystyle\sum_{i=0}^{i^*}b(\{i\}) + i^* + 1$ and all but $i^*$ of the elements are maximal in $I$.
\end{lemma}

\begin{proof}
The ideal $I$ whose maximal elements are $B(\varnothing) \cup B(\{1\}) \cup \dots \cup B(\{i^*\}) \cup \{x\}$ has $\sum_{i=0}^{i^*}b(\{i\}) + i^* + 1$ elements, all but $i^*$ of which are maximal.  Suppose there were another ideal in $P$ whose set of maximal elements $S$ had the same size.  It must be the case that $S$ contains an element greater than $i^*$ and an element greater than $x$, but this would imply that the ideal has at least $i^*+1$ elements that are not maximal.

Now from Corollary~\ref{P without I}, $K_{P \setminus I}(\mathbf x)$ is determined by $\KP$. In particular, the up-jumps of the minimal elements of $P \setminus I$ are determined. Since the minimal elements are $B(\{x\}) \cup B(\{i^*,x\})$ and $i^*+1$ if it exists (which has up-jump $\lambda_1-i^*-1$), the result follows.
\end{proof}

%Observe that the minimal elements of $P \backslash I$ are $\{i^*+1\} \cup B(\{i^*, x\}) \cup B(\{x\})$.  Also note that Corollary \ref{P without I} tells us that $K_{P\backslash I}(\mathbf{x})$ is determined by $\KP$.  

We are now ready to prove the main theorem of this section.

\begin{thm}
If $\sh(P) = \lambda = (\lambda_1, 2, 1, 1, \dots, 1)$ is nearly hook shaped, then $P$ is uniquely determined by $\KP$.
\end{thm}

\begin{proof}
First we will assume that $P$ is Type 1. We will induct on the size of $P^+$.  If $|P^+| = 0$, then Lemma \ref{below antichain} implies the result.

Now suppose the statement holds for all Type 1 posets with smaller $P^+$.  Let $\hat{P}$ be $P$ with its maximal elements removed, which we can determine from $\KP$.  We will show that there is a unique way to recover $P$ from $\hat{P}$ given $\KP$.  In order to show this, we need to consider the case when $\sh(\hat{P})$ is hook shaped and when it is nearly hook shaped.

If $\sh(\hat{P})$ is hook shaped, then every element that covers $x$ in $A$ must be maximal in $P$. Given $P^-$ and Lemma \ref{upjump of x}, we know which element of the chain must cover and be covered by $x$ in $P$, so we can find an element in $\hat P$ that corresponds to $x$ in $P$. Since we know $P^-$, and Lemma \ref{below antichain} tells us the number of elements in $P$ that cover any ideal in $P^-$, there is a unique way to add the missing elements of $A$ to $\hat P$. We also add a new maximal element to the top of the chain that covers all the maximal elements of $\hat P$ (except possibly $x$).  The only other relation that can occur in $P$ is that $x$ may also be covered by this final maximal element at the top of $P^+$, which we can again determine from Lemma \ref{upjump of x}.

If $\sh(\hat{P})$ is nearly hook shaped, then to get $P$ from $\hat P$, we must add a maximal element that covers all of the maximal elements of $\hat P$, then add elements to the longest antichain of $\hat P$ until Lemma~\ref{below antichain} is satisfied. However, there may be some ambiguity if there is an automorphism of $P^-$ that switches $i^*$ and $x$ that does not extend to an automorphism of $\hat P$, and $b(\{i^*\}) \neq b(\{x\})$ in $P$. In this case, the multiset of up-jump values of elements of $B(\{i^*\})$ must differ from that of $B(\{x\})$, so Lemma~\ref{Near Hook P minus I} is enough to distinguish $x$ from $i^*$.

Therefore, the result holds when $P$ is Type 1, as well as for Type 2 since the dual of a Type 2 poset is Type 1. Finally, if $P$ is Type 3, then $P$ can be expressed as a union of a chain and an antichain $A$ (which do not intersect). As in the proof of Theorem~\ref{Hook Shaped}, $P$ is then determined by the jump pairs of its elements, which can be determined from $\KP$ by Theorem~\ref{jumppair}.
 %Each element of $A$ covers at most one element of the chain and is covered by at most one element of the chain.  For each $a \in A$, $\jumppair(a)$ is determined by the element that $a$ covers and the element that covers $a$.  This implies that Type 3 posets are determined by the jump pairs of their elements.  Since the multiset of $\jumppair(a)$ for all $a$ is determined by $\supp_L(\KP)$, it follows from Theorem \ref{jumppair} that if $P$ is Type 3, then P is uniquely determined by $\KP$.
\end{proof}

\begin{example}
Suppose $b(\{i^*\}) = 4$ and $b(\{x\}) = 3$ and the poset $\hat{P}$ is shown below.
% and $P\backslash I$ are the following posets, where $I$ is the ideal in Lemma~\ref{Near Hook P minus I}:
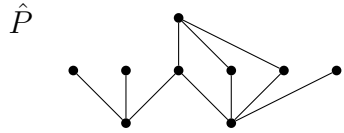
\begin{center}
\begin{tikzpicture} %P hat
 [auto,
 vertex/.style={circle,draw=black!100,fill=black!100,thick,inner sep=0pt,minimum size=1mm}, scale =.7]
\node(P) at (-3, 2) {$\hat{P}$};
\node (1) at (-1,0) [vertex] {};
\node (2) at (1,0) [vertex] {};
\node (3) at (0,1) [vertex] {};
\node (4) at (-1,1) [vertex] {};
\node (5) at (-2,1) [vertex] {};
\node (6) at (-3,1) {};
\node (7) at (1,1) [vertex] {};
\node (8) at (2,1) [vertex] {};
\node (9) at (3,1) [vertex] {};
\node (10) at (0,2) [vertex] {};
\node (12) at (4, 1) {};
\draw [-] (1) to (3);
\draw [-] (1) to (4);
\draw [-] (1) to (5);
\draw [-] (2) to (3);
\draw [-] (2) to (7);
\draw [-] (2) to (8);
\draw [-] (2) to (9);
\draw [-] (3) to (10);
\draw [-] (7) to (10);
\draw [-] (8) to (10);
\end{tikzpicture}
%\hspace{3 cm}
%\begin{tikzpicture} %P minus I
% [auto,
% vertex/.style={circle,draw=black!100,fill=black!100,thick,inner sep=0pt,minimum size=1mm}, scale =.7]
%\node(P) at (-3, 3) {$P \backslash I$};
%\node (3) at (0,1) [vertex] {};
%\node (4) at (-1,1) [vertex] {};
%\node (5) at (-2,1) [vertex] {};
%\node (6) at (-3,1) [vertex] {};
%\node (7) at (1,1) {};
%\node (8) at (2,1) {};
%\node (9) at (3,1) {};
%\node (10) at (0,2) [vertex] {};
%\node (11) at (0,3) [vertex] {};
%\node (12) at (4, 1) {};
%\draw [-] (3) to (10);
%\draw [-] (10) to (11);
%\draw [-] (4) to (11);
%\draw [-] (5) to (11);
%\end{tikzpicture}
\end{center}
Note that there is an automorphism of $P^-$ that switches the two minimal elements. In order to determine $P$ from $\hat{P}$, we need to determine which of the minimal elements is $x$ and which one is $i^*$.  The following two posets are both formed by adding maximal elements to $\hat{P}$, and they both satisfy $b(\{i^*\}) = 4$, $b(\{x\}) = 3$.  %The posets $P_1$ and $P_2$ each have a unique 6 element ideal with 5 maximal elements.  We will call these ideals $I_1$ and $I_2$.

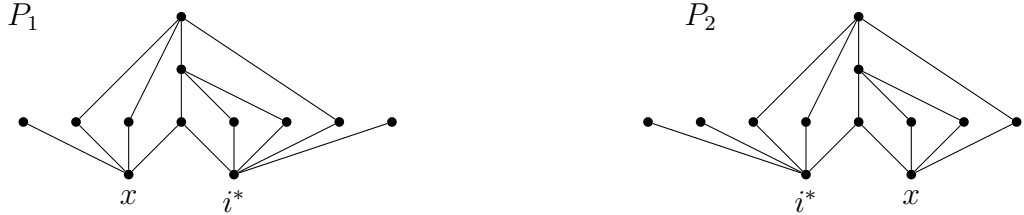
\begin{center}
\begin{tikzpicture} %Potential posets
 [auto,
 vertex/.style={circle,draw=black!100,fill=black!100,thick,inner sep=0pt,minimum size=1mm}, scale =.7]
\node(P) at (-3, 3) {$P_1$};
\node (1) at (-1,0) [vertex,label=below:$x$] {};
\node (2) at (1,0) [vertex,label=below:$i^*$] {};
\node (3) at (0,1) [vertex] {};
\node (4) at (-1,1) [vertex] {};
\node (5) at (-2,1) [vertex] {};
\node (6) at (-3,1) [vertex] {};
\node (7) at (1,1) [vertex] {};
\node (8) at (2,1) [vertex] {};
\node (9) at (3,1) [vertex] {};
\node (10) at (0,2) [vertex] {};
\node (11) at (0,3) [vertex] {};
\node (12) at (4, 1) [vertex] {};
\draw [-] (1) to (3);
\draw [-] (1) to (4);
\draw [-] (1) to (5);
\draw [-] (1) to (6);
\draw [-] (2) to (3);
\draw [-] (2) to (7);
\draw [-] (2) to (8);
\draw [-] (2) to (9);
\draw [-] (2) to (12);
\draw [-] (3) to (10);
\draw [-] (10) to (11);
\draw [-] (4) to (11);
\draw [-] (5) to (11);
\draw [-] (7) to (10);
\draw [-] (8) to (10);
\draw [-] (9) to (11);
\end{tikzpicture}
\hspace{3 cm}
\begin{tikzpicture}
 [auto,
 vertex/.style={circle,draw=black!100,fill=black!100,thick,inner sep=0pt,minimum size=1mm}, scale =.7]
\node(P) at (-3, 3) {$P_2$};
\node (1) at (-1,0) [vertex,label=below:$i^*$] {};
\node (2) at (1,0) [vertex,label=below:$x$] {};
\node (3) at (0,1) [vertex] {};
\node (4) at (-1,1) [vertex] {};
\node (5) at (-2,1) [vertex] {};
\node (6) at (-3,1) [vertex] {};
\node (7) at (1,1) [vertex] {};
\node (8) at (2,1) [vertex] {};
\node (9) at (3,1) [vertex] {};
\node (10) at (0,2) [vertex] {};
\node (11) at (0,3) [vertex] {};
\node (12) at (-4, 1) [vertex] {};
\draw [-] (1) to (3);
\draw [-] (1) to (4);
\draw [-] (1) to (5);
\draw [-] (1) to (6);
\draw [-] (1) to (12);
\draw [-] (2) to (3);
\draw [-] (2) to (7);
\draw [-] (2) to (8);
\draw [-] (2) to (9);
\draw [-] (3) to (10);
\draw [-] (10) to (11);
\draw [-] (4) to (11);
\draw [-] (5) to (11);
\draw [-] (7) to (10);
\draw [-] (8) to (10);
\draw [-] (9) to (11);
\end{tikzpicture}
\end{center}
However, in $P_1$, the up-jump values for elements of $B(\{x\})$ are $\{0,1,1\}$, while in $P_2$, they are $\{1,2,2\}$. Thus we can distinguish these two cases by Lemma~\ref{Near Hook P minus I}.
%Observe that $P_1\backslash I_1 = P \backslash I$ but $P_2\backslash I_2 \neq P \backslash I$, so it follows that $P_1 = P$.
\end{example}

In summary, we have shown that if $\sh(P) = (\lambda_1, \lambda_2)$, $\sh(P) = (\lambda_1, 1, \dots, 1)$ or $\sh(P) = (\lambda_1, 2, 1, \dots, 1)$, then $P$ is uniquely determined by its $P$-partition generating function.

For most of the remaining shapes, we present a negative result in the next section.

\section{Posets with the same $P$-partition generating function}
In this section, we give a method for constructing distinct posets with the same partition generating function.

\begin{defn}
Suppose that $P$ and $Q$ are finite posets.  If $\KP = \KQ$, then we say that $P$ and $Q$ are \emph{$K$-equivalent}.
\end{defn}

Given a poset $P$ and a pair of incomparable elements $(x, y)$, write $P + (x \prec y)$ for the poset obtained by adding the relation $x \prec y$ to $P$ (and taking the transitive closure).

\begin{lemma}\label{Kequiv}
Suppose $R$ is a finite poset and $\phi\colon R \rightarrow R$ is an automorphism.  Let $e = (e_1, e_2)$ and $f = (f_1, f_2)$ be two pairs of incomparable elements in $R$ such that in $R+(f_2 \prec f_1)$, both $e_1 \prec e_2$ and $\phi^{-1}(e_1) \prec \phi^{-1}(e_2)$. If $m>0$ is the smallest positive integer such that $\phi^{m+1}(e)=e$, then
\begin{align*}
P &= R + (f_1 \prec f_2) + (e_1 \prec e_2) + (\phi(e_1) \prec \phi(e_2)) + \dots + (\phi^{m-1}(e_1) \prec \phi^{m-1}(e_2)),\\
Q &= R + (f_1 \prec f_2) + (\phi(e_1) \prec \phi(e_2)) + (\phi^2(e_1) \prec \phi^2(e_2))+ \dots + (\phi^{m}(e_1) \prec \phi^{m}(e_2)).
\end{align*}
are $K$-equivalent (assuming both are naturally labeled).
\end{lemma}

\begin{proof}
Let
\[S = R + (e_1 \prec e_2) + (\phi(e_1) \prec \phi(e_2)) + \dots + (\phi^{m-1}(e_1) \prec \phi^{m-1}(e_2)).\] Every partition of $S$ is either a partition of $P$ or a partition of $S + (f_2 \prec f_1)$ (which is not naturally labeled), so $K_S(\mathbf{x}) = \KP + K_{S + (f_2 \prec f_1)}(\mathbf{x})$. Solving for $\KP$ gives
\[\KP = K_S(\mathbf{x}) - K_{S + (f_2 \prec f_1)}(\mathbf{x}).\]
Similarly, let 
\[S'=R + (\phi(e_1) \prec \phi(e_2)) + \dots + (\phi^{m}(e_1) \prec \phi^{m}(e_2)).\]
The partitions of $Q$ are the partitions of $S'$ with the partitions of $S' + (f_2 \prec f_1)$ removed, so
\[\KQ = K_{S'}(\mathbf{x}) - K_{S' + (f_2 \prec f_1)}(\mathbf{x}).\]
Observe that $S \cong S'$ since $S' = \phi(S)$, so $S$ and $S'$ are trivially $K$-equivalent.

By assumption, $(e_1 \prec e_2)$ and $(\phi^m(e_1) \prec \phi^m(e_2))$ in $R + (f_2 \prec f_1)$.  It follows that 
\begin{align*}
S + (f_2 \prec f_1) & = R + (f_2 \prec f_1) + (e_1 \prec e_2) + (\phi(e_1) \prec \phi(e_2)) + \dots + (\phi^{m-1}(e_1) \prec \phi^{m-1}(e_2)) \\
&= R + (f_2 \prec f_1) + (\phi(e_1) \prec \phi(e_2)) + \dots + (\phi^{m-1}(e_1) \prec \phi^{m-1}(e_2))
\end{align*}
and
\begin{align*}
S' + (f_2 \prec f_1) &= R + (f_2 \prec f_1) + (\phi(e_1) \prec \phi(e_2)) + \dots  + (\phi^{m}(e_1) \prec \phi^{m}(e_2)) \\
&= R + (f_2 \prec f_1) + (\phi(e_1) \prec \phi(e_2)) + \dots + (\phi^{m-1}(e_1) \prec \phi^{m-1}(e_2)).
\end{align*}
These are the same poset, so $K_{S+(f_2 \prec f_1)}(\mathbf x) = K_{S' + (f_2 \prec f_1)}(\mathbf x)$.  Therefore $P$ and $Q$ are $K$-equivalent.
\end{proof}

We will now give some examples of posets that can be shown to be $K$-equivalent by using the previous lemma.
\begin{example}\label{sevenelement}
Consider the following $7$-element posets.  These posets are not isomorphic but they are $K$-equivalent.
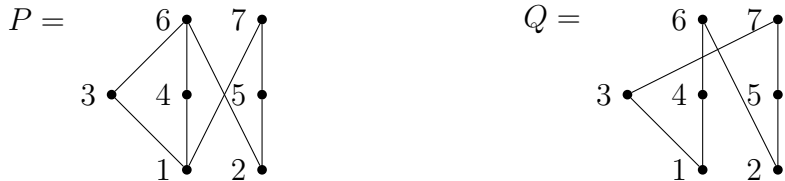
\begin{center}
\begin{tikzpicture}
 [auto,
 vertex/.style={circle,draw=black!100,fill=black!100,thick,inner sep=0pt,minimum size=1mm}]
\node(P) at (-2, 2) {$P =$};
\node (v1) at ( 0,0) [vertex,label=left:$1$] {};
\node (v2) at ( 1,0) [vertex,label=left:$2$] {};
\node (v3) at ( -1,1) [vertex,label=left:$3$] {};
\node (v4) at ( 0,1) [vertex,label=left:$4$] {};
\node (v5) at ( 1,1) [vertex,label=left:$5$] {};
\node (v6) at ( 0,2) [vertex,label=left:$6$] {};
\node (v7) at ( 1,2) [vertex,label=left:$7$] {};
\draw [-] (v1) to (v4);
\draw [-] (v4) to (v6);
\draw [-] (v2) to (v5);
\draw [-] (v5) to (v7);
\draw [-] (v1) to (v3);
\draw [-] (v3) to (v6);
\draw [-] (v1) to (v7);
\draw [-] (v2) to (v6);
\end{tikzpicture}
\hspace{3 cm}
\begin{tikzpicture}
 [auto,
 vertex/.style={circle,draw=black!100,fill=black!100,thick,inner sep=0pt,minimum size=1mm}]
\node(Q) at (2, 2) {$Q =$};
\node (v1) at ( 4,0) [vertex,label=left:$1$] {};
\node (v2) at ( 5,0) [vertex,label=left:$2$] {};
\node (v3) at ( 3,1) [vertex,label=left:$3$] {};
\node (v4) at ( 4,1) [vertex,label=left:$4$] {};
\node (v5) at ( 5,1) [vertex,label=left:$5$] {};
\node (v6) at ( 4,2) [vertex,label=left:$6$] {};
\node (v7) at ( 5,2) [vertex,label=left:$7$] {};
\draw [-] (v1) to (v4);
\draw [-] (v4) to (v6);
\draw [-] (v2) to (v5);
\draw [-] (v5) to (v7);
\draw [-] (v1) to (v3);
\draw [-] (v3) to (v7);
\draw [-] (v2) to (v6);
\end{tikzpicture}
\end{center}

We can express $P$ and $Q$ in terms of a subposet $R$ with a nontrivial automorphism along with some additional covering relations.

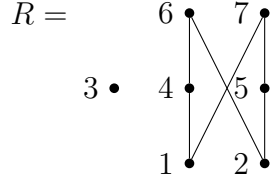
\begin{center}
\begin{tikzpicture}
 [auto,
 vertex/.style={circle,draw=black!100,fill=black!100,thick,inner sep=0pt,minimum size=1mm}]
\node(R) at (-2, 2) {$R =$};
\node (v1) at ( 0,0) [vertex,label=left:$1$] {};
\node (v2) at ( 1,0) [vertex,label=left:$2$] {};
\node (v3) at ( -1,1) [vertex,label=left:$3$] {};
\node (v4) at ( 0,1) [vertex,label=left:$4$] {};
\node (v5) at ( 1,1) [vertex,label=left:$5$] {};
\node (v6) at ( 0,2) [vertex,label=left:$6$] {};
\node (v7) at ( 1,2) [vertex,label=left:$7$] {};
\draw [-] (v1) to (v4);
\draw [-] (v4) to (v6);
\draw [-] (v2) to (v5);
\draw [-] (v5) to (v7);
\draw [-] (v1) to (v7);
\draw [-] (v2) to (v6);
\end{tikzpicture}
\end{center}
The automorphism $\phi$ is the map that fixes $3$ and swaps the two chains.  Let $e = (3, 6)$, $\phi(e) = (3, 7)$, and $f = (1, 3)$.   Since both $3 \prec 6$ and $3 \prec 7$ in $R + (3 \prec 1)$, it follows from  Lemma \ref{Kequiv} that $\KP = \KQ$.
\end{example}

\begin{example} \label{cycle}
Consider the following nonisomorphic $8$-element posets.

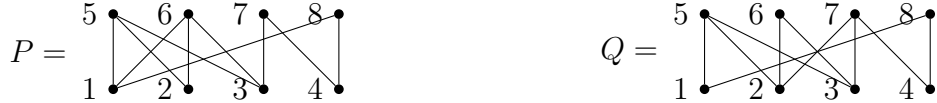
\begin{center}
\begin{tikzpicture}
 [auto,
 vertex/.style={circle,draw=black!100,fill=black!100,thick,inner sep=0pt,minimum size=1mm}]
\node(P) at (-1, .5) {$P =$};
\node (1) at ( 0,0) [vertex,label=left:$1$] {};
\node (2) at ( 1,0) [vertex,label=left:$2$] {};
\node (3) at ( 2,0) [vertex,label=left:$3$] {};
\node (4) at ( 3,0) [vertex,label=left:$4$] {};
\node (5) at ( 0,1) [vertex,label=left:$5$] {};
\node (6) at ( 1,1) [vertex,label=left:$6$] {};
\node (7) at ( 2,1) [vertex,label=left:$7$] {};
\node (8) at (3, 1) [vertex,label=left:$8$] {};
\draw [-] (1) to (5);
\draw [-] (1) to (8);
\draw [-] (2) to (5);
\draw [-] (2) to (6);
\draw [-] (3) to (6);
\draw [-] (3) to (7);
\draw [-] (4) to (7);
\draw [-] (4) to (8);
\draw [-] (1) to (6);
\draw [-] (3) to (5);
\end{tikzpicture}
\hspace{3 cm}
\begin{tikzpicture}
 [auto,
 vertex/.style={circle,draw=black!100,fill=black!100,thick,inner sep=0pt,minimum size=1mm}]
\node(P=Q) at (-1, .5) {$Q =$};
\node (1) at ( 0,0) [vertex,label=left:$1$] {};
\node (2) at ( 1,0) [vertex,label=left:$2$] {};
\node (3) at ( 2,0) [vertex,label=left:$3$] {};
\node (4) at ( 3,0) [vertex,label=left:$4$] {};
\node (5) at ( 0,1) [vertex,label=left:$5$] {};
\node (6) at ( 1,1) [vertex,label=left:$6$] {};
\node (7) at ( 2,1) [vertex,label=left:$7$] {};
\node (8) at (3, 1) [vertex,label=left:$8$] {};
\draw [-] (1) to (5);
\draw [-] (1) to (8);
\draw [-] (2) to (5);
\draw [-] (2) to (6);
\draw [-] (3) to (6);
\draw [-] (3) to (7);
\draw [-] (4) to (7);
\draw [-] (4) to (8);
\draw [-] (2) to (7);
\draw [-] (3) to (5);
\end{tikzpicture}
\end{center}
The poset $R$ shown below has an automorphism $\phi$ given by the permutation $(1234)(5678)$.
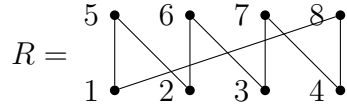
\begin{center}
\begin{tikzpicture}
 [auto,
 vertex/.style={circle,draw=black!100,fill=black!100,thick,inner sep=0pt,minimum size=1mm}]
\node(R) at (-1, .5) {$R =$};
\node (1) at ( 0,0) [vertex,label=left:$1$] {};
\node (2) at ( 1,0) [vertex,label=left:$2$] {};
\node (3) at ( 2,0) [vertex,label=left:$3$] {};
\node (4) at ( 3,0) [vertex,label=left:$4$] {};
\node (5) at ( 0,1) [vertex,label=left:$5$] {};
\node (6) at ( 1,1) [vertex,label=left:$6$] {};
\node (7) at ( 2,1) [vertex,label=left:$7$] {};
\node (8) at (3, 1) [vertex,label=left:$8$] {};
\draw [-] (1) to (5);
\draw [-] (1) to (8);
\draw [-] (2) to (5);
\draw [-] (2) to (6);
\draw [-] (3) to (6);
\draw [-] (3) to (7);
\draw [-] (4) to (7);
\draw [-] (4) to (8);
\end{tikzpicture}
\end{center}
Let $e = (1, 6)$, $\phi(e) = (2, 7)$, and $f = (3, 5)$.  Since both $1 \prec 6$ and $2 \prec 7$ in $R + (5 \prec 3)$, it follows from  Lemma \ref{Kequiv} that $\KP = \KQ$.
\end{example}

Observe that the posets in Example \ref{sevenelement} have shape $(3, 3, 1)$ and the posets in Example \ref{cycle} have shape $(2, 2, 2, 2)$.  We can generalize these examples to construct pairs of posets of any larger shape that are $K$-equivalent.

\begin{thm}
For all partitions $\lambda$ with $\lambda \supset (3, 3, 1)$ or $\lambda \supset (2, 2, 2, 2)$, there exist posets $P$ and $Q$ such that $P \ncong Q$, $\sh(P)=\sh(Q) = \lambda$, and $\KP = \KQ$.
\end{thm}

\begin{proof}
We will prove this result by building off of the posets from Example \ref{sevenelement} and Example \ref{cycle}.  Observe that if $\sh(P) = \mu = (\mu_1, \dots, \mu_k)$ and $\sh(Q) = \nu = (\nu_1, \dots, \nu_l)$, then
\[\sh(P \oplus Q) = \mu + \nu = (\mu_1+\nu_1, \mu_2+\nu_2, \dots ).\] 
Also observe that
\[\sh(P \sqcup Q) = \mu \cup \nu = (\mu_1'+\nu_1', \mu_2'+\nu_2', \dots )'.\]

Let $\lambda = (\lambda_1, \lambda_2, \dots, \lambda_k)$ be a partition that contains either $(3, 3, 1)$ or $(2, 2, 2, 2)$.  If $\lambda$ contains $(3, 3, 1)$ then we will first form two $K$-equivalent posets that have shape $(\lambda_1, \lambda_2, \lambda_3)$ and then take the disjoint union with the poset of disjoint chains of sizes $\lambda_4, \lambda_5, \dots, \lambda_k$.

Consider the following posets $P'$ and $Q'$ depicted below of shape $(\lambda_1, \lambda_2, \lambda_3)$.

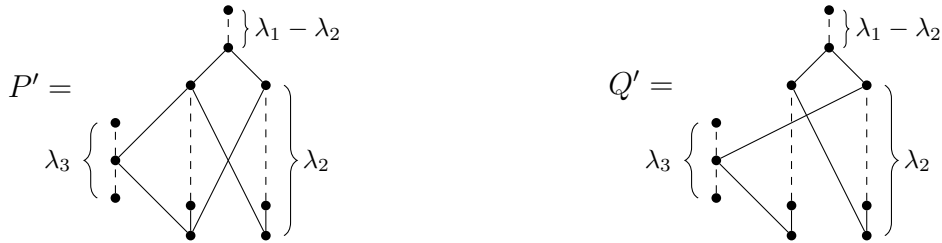
\begin{center}
\begin{tikzpicture}
 [auto,
 vertex/.style={circle,draw=black!100,fill=black!100,thick,inner sep=0pt,minimum size=1mm}]
\node(P) at (-2, 2) {$P' =$};
\node (v1) at ( 0,0) [vertex,label=left:$$] {};
\node (v2) at ( 1,0) [vertex,label=left:$$] {};
\node (v3) at ( -1,1) [vertex,label=left:$$] {};
\node (v4) at ( 0,.4) [vertex,label=left:$$] {};
\node (v5) at ( 1,.4) [vertex,label=left:$$] {};
\node (v6) at ( 0,2) [vertex,label=left:$$] {};
\node (v7) at ( 1,2) [vertex,label=left:$$] {};
\node (v8) at (-1,.5) [vertex] {};
\node (v9) at (-1,1.5) [vertex] {};
\node (v10) at (.5, 2.5) [vertex] {};
\node (v11) at (.5, 3) [vertex] {};
\draw [-] (v1) to (v4);
\draw [dashed] (v4) to (v6);
\draw [-] (v2) to (v5);
\draw [dashed] (v5) to (v7);
\draw [-] (v1) to (v3);
\draw [-] (v3) to (v6);
\draw [-] (v1) to (v7);
\draw [-] (v2) to (v6);
\draw [dashed] (v8) to (v3);
\draw [dashed] (v3) to (v9);
\draw [-] (v6) to (v10);
\draw [-] (v7) to (v10);
\draw [dashed] (v10) to (v11);
\draw [decorate,decoration={brace,amplitude=5pt},xshift=-4pt,yshift=0pt]
(-1.1,.5) -- (-1.1,1.5) node [black,midway,xshift=-0.2cm] 
{\footnotesize $\lambda_3$};
\draw [decorate,decoration={brace,amplitude=5pt,mirror,raise=4pt},yshift=0pt]
(1.1,0) -- (1.1,2) node [black,midway,xshift=0.9cm] {\footnotesize
$\lambda_2$};
\draw [decorate,decoration={brace,amplitude=2pt,mirror,raise=4pt},yshift=0pt]
(.55,2.5) -- (.55,3) node [black,midway,xshift=1.6cm] {\footnotesize
$\lambda_1 - \lambda_2$};
\end{tikzpicture}
\hspace{3 cm}
\begin{tikzpicture}
 [auto,
 vertex/.style={circle,draw=black!100,fill=black!100,thick,inner sep=0pt,minimum size=1mm}]
\node(Q) at (-2, 2) {$Q' =$};
\node (v1) at ( 0,0) [vertex,label=left:$$] {};
\node (v2) at ( 1,0) [vertex,label=left:$$] {};
\node (v3) at ( -1,1) [vertex,label=left:$$] {};
\node (v4) at ( 0,.4) [vertex,label=left:$$] {};
\node (v5) at ( 1,.4) [vertex,label=left:$$] {};
\node (v6) at ( 0,2) [vertex,label=left:$$] {};
\node (v7) at ( 1,2) [vertex,label=left:$$] {};
\node (v8) at (-1,.5) [vertex] {};
\node (v9) at (-1,1.5) [vertex] {};
\node (v10) at (.5, 2.5) [vertex] {};
\node (v11) at (.5, 3) [vertex] {};
\draw [-] (v1) to (v4);
\draw [dashed] (v4) to (v6);
\draw [-] (v2) to (v5);
\draw [dashed] (v5) to (v7);
\draw [-] (v1) to (v3);
\draw [-] (v3) to (v7);
\draw [-] (v2) to (v6);
\draw [dashed] (v8) to (v3);
\draw [dashed] (v3) to (v9);
\draw [-] (v6) to (v10);
\draw [-] (v7) to (v10);
\draw [dashed] (v10) to (v11);
\draw [decorate,decoration={brace,amplitude=5pt},xshift=-4pt,yshift=0pt]
(-1.1,.5) -- (-1.1,1.5) node [black,midway,xshift=-0.2cm] 
{\footnotesize $\lambda_3$};
\draw [decorate,decoration={brace,amplitude=5pt,mirror,raise=4pt},yshift=0pt]
(1.1,0) -- (1.1,2) node [black,midway,xshift=0.9cm] {\footnotesize
$\lambda_2$};
\draw [decorate,decoration={brace,amplitude=2pt,mirror,raise=4pt},yshift=0pt]
(.55,2.5) -- (.55,3) node [black,midway,xshift=1.6cm] {\footnotesize
$\lambda_1 - \lambda_2$};
\end{tikzpicture}
\end{center}

Since $\lambda_2 \geq 3$, $P' \ncong Q'$.  As in Example~\ref{sevenelement}, it follows from Lemma \ref{Kequiv} that $P'$ and $Q'$ are $K$-equivalent. Now let $R$ be the poset of disjoint chains of sizes $\lambda_4, \lambda_5, \dots, \lambda_k$, and let $P = P' \sqcup R$ and $Q = Q' \sqcup R$.  These posets have the desired shape $\lambda$, and since $K_{P'}(\mathbf{x}) = K_{Q'}(\mathbf{x})$, it follows that $\KP = \KQ$.

Now suppose $\lambda \supset (2, 2, 2, 2)$ but it does not contain $(3, 3, 1)$, so $\lambda$ has the form $\lambda = (\lambda_1, 2^j, 1^l)$.  Let $C$ be a $(\lambda_1 - 2)$-element chain, and let $R$ be the poset with $j-3$ disjoint $2$-element chains and $l$ disjoint single elements.  Let $P_8$ and $Q_8$ be the $8$-element posets from Example \ref{cycle}.  If we let $P = (P_8 \oplus C) \sqcup R$ and $Q = (Q_8 \oplus C) \sqcup R$, then $P$ and $Q$ have the desired shape. Since $P_8$ and $Q_8$ are $K$-equivalent, $\KP = \KQ$.
\end{proof}

The only remaining shapes for which it is not known whether there exists non-isomorphic $K$-equivalent posets are those of the form $(\lambda_1, 2, 2, 1, 1, \dots, 1)$.

\printbibliography

\end{document}